\documentclass[10pt,a4paper]{article} 

\usepackage[T1]{fontenc}
\usepackage[utf8]{inputenc}
\usepackage{calc}
\usepackage{indentfirst}
\usepackage{fancyhdr}
\usepackage{graphicx,epstopdf}
\usepackage{lastpage}
\usepackage{ifthen}
\usepackage{float}
\usepackage{amsmath}
\usepackage{setspace}
\usepackage{enumitem}
\usepackage{booktabs} 
\usepackage[largestsep]{titlesec}
\usepackage{etoolbox} 
\usepackage{tabto} 
\usepackage{xcolor} 
\usepackage{soul} 

\usepackage{multirow}
\usepackage{microtype} 
\usepackage{tikz} 
\usepackage{totcount} 
\usepackage{amsthm}

\usepackage[left=2.7cm,
	right=2.7cm,
	top=1.8cm,
	bottom=1.5cm,
	includehead,
	includefoot]{geometry}

 \newcounter{theorem}
 \newcounter{lemma}
 \newcounter{corollary}
 \newtheorem{Corollary}[corollary]{Corollary}
 
 \newcounter{proposition}
 \newtheorem{Proposition}[proposition]{Proposition}
 
 \newcounter{characterization}
 \newcounter{property}
 \newcounter{problem}
 \newcounter{example}
 \newcounter{examplesanddefinitions}
 \newcounter{remark}
 \newtheorem{Remark}[remark]{Remark}
 
 \newcounter{definition}
 \newtheorem{Definition}[definition]{Definition}
 
 \newcounter{hypothesis}
 \newcounter{notation}
\usepackage[capitalize,nameinlink]{cleveref}
\crefformat{equation}{(#2#1#3)}
\crefmultiformat{equation}{(#2#1#3)}{ and~(#2#1#3)}{, (#2#1#3)}{ and~(#2#1#3)}
\usepackage{bm}
\usepackage{amssymb,pifont} 
\newcommand{\cmark}{\ding{51}}
\newcommand{\xmark}{\ding{55}}
\usepackage{multirow}
\usepackage{blkarray}
\newcommand{\idxsize}{\scriptstyle}
\usepackage{multirow}

\providecommand{\keywords}[1]
{
  \small	
  \textbf{\textit{Keywords---}} #1
}
\providecommand{\funding}[1]
{
  \small	
  \textbf{\textit{Funding---}} #1
}
\providecommand{\acknowledgments}[1]
{
  \small	
  \textbf{\textit{Acknowledgments---}} #1
}

\title{Symplectic Model Order Reduction\\ with Non-Orthonormal Bases}

\author{Patrick Buchfink $^{1,}$\footnote{\{patrick.buchfink,haasdonk\}@ians.uni-stuttgart.de}, Ashish Bhatt$^{2,}$\footnote{ashishbhatt@iitism.ac.in}~and Bernard Haasdonk $^{1,*}$ }

\date{\parbox{\linewidth}{\centering%
 $ ^{1}$ Institute of Applied Analysis and Numerical Simulation,\\ University of Stuttgart, 70569 Stuttgart, Germany \endgraf\bigskip
 $ ^{2}$ Indian Institute of Technology (ISM), Dhanbad, Jharkhand 826004, India \endgraf\bigskip
  \today\endgraf}}

\newcommand{\N}{\ensuremath{\mathbb{N}}}
\newcommand{\R}{\ensuremath{\mathbb{R}}}
\newcommand{\Rpos}{\ensuremath{\mathbb{R}_{>0}}}

\newcommand{\Cn}{\ensuremath{\mathbb{C}}}
\newcommand{\imagUnit}{\ensuremath{{\mathrm{i}}}}

\newcommand{\txtfrac}[2]{\ensuremath{{#1}/{#2}}}

\newcommand{\lb}{\left(}
\newcommand{\rb}{\right)}
\newcommand{\lsb}{\left[}
\newcommand{\rsb}{\right]}
\newcommand{\lcb}{\left\{}
\newcommand{\rcb}{\right\}}

\DeclareMathOperator*{\argmax}{\operatornamewithlimits{argmax}}
\let\Re\relax
\DeclareMathOperator{\Re}{Re}
\newcommand{\Reb}[1]{\ensuremath{\Re\lb{#1}\rb}}
\let\Im\relax
\DeclareMathOperator{\Im}{Im}
\newcommand{\Imb}[1]{\ensuremath{\Im\lb{#1}\rb}}
\DeclareMathOperator*{\rank}{rank}

\DeclareMathOperator*{\img}{img}
\DeclareMathOperator*{\trace}{trace}
\DeclareMathOperator*{\colspan}{colspan}
\DeclareMathOperator*{\diag}{diag}

\newcommand{\colspanb}[1]{\ensuremath{\colspan{\lb #1 \rb}}}
\newcommand{\traceb}[1]{\trace\lb{#1}\rb}

\newcommand{\abs}[1]{\ensuremath{\left|{#1}\right|}}
\newcommand{\minimize}[1]{\underset{#1}{\operatorname{minimize}}}

\newcommand{\fb}{\ensuremath{\bm{b}}}

\newcommand{\fe}{\ensuremath{\bm{e}}}
\newcommand{\ff}{\ensuremath{\bm{f}}}
\newcommand{\fg}{\ensuremath{\bm{g}}}
\newcommand{\fh}{\ensuremath{\bm{h}}}

\newcommand{\fp}{\ensuremath{\bm{p}}}
\newcommand{\fq}{\ensuremath{\bm{q}}}
\newcommand{\fr}{\ensuremath{\bm{r}}}
\newcommand{\fs}{\ensuremath{\bm{s}}}
\newcommand{\ft}{\ensuremath{\bm{t}}}
\newcommand{\fu}{\ensuremath{\bm{u}}}
\newcommand{\fv}{\ensuremath{\bm{v}}}

\newcommand{\fx}{\ensuremath{\bm{x}}}
\newcommand{\fy}{\ensuremath{\bm{y}}}

\newcommand{\fA}{\ensuremath{\bm{A}}}
\newcommand{\fB}{\ensuremath{\bm{B}}}
\newcommand{\fC}{\ensuremath{\bm{C}}}
\newcommand{\fD}{\ensuremath{\bm{D}}}
\newcommand{\fE}{\ensuremath{\bm{E}}}
\newcommand{\fF}{\ensuremath{\bm{F}}}
\newcommand{\fG}{\ensuremath{\bm{G}}}
\newcommand{\fH}{\ensuremath{\bm{H}}}
\newcommand{\fI}{\ensuremath{\bm{I}}}

\newcommand{\fK}{\ensuremath{\bm{K}}}

\newcommand{\fM}{\ensuremath{\bm{M}}}
\newcommand{\fN}{\ensuremath{\bm{N}}}

\newcommand{\fP}{\ensuremath{\bm{P}}}
\newcommand{\fQ}{\ensuremath{\bm{Q}}}
\newcommand{\fR}{\ensuremath{\bm{R}}}
\newcommand{\fS}{\ensuremath{\bm{S}}}

\newcommand{\fU}{\ensuremath{\bm{U}}}
\newcommand{\fV}{\ensuremath{\bm{V}}}
\newcommand{\fW}{\ensuremath{\bm{W}}}
\newcommand{\fX}{\ensuremath{\bm{X}}}
\newcommand{\fY}{\ensuremath{\bm{Y}}}

\newcommand{\fmu}{\ensuremath{\bm{\mu}}}

\newcommand{\fxi}{\ensuremath{\bm{\xi}}}

\newcommand{\fphi}{\ensuremath{\bm{\varphi}}}

\newcommand{\fSigma}{\ensuremath{\bm{\Sigma}}}

\newcommand{\fPhi}{\ensuremath{\bm{\varPhi}}}
\newcommand{\fPsi}{\ensuremath{\bm{\Psi}}}

\newcommand{\fzero}{\ensuremath{\bm{0}}}

\newcommand{\rT}[1]{\ensuremath{#1^{\textsf{T}}}}
\newcommand{\rTb}[1]{\ensuremath{\rT{\lb#1\rb}}}
\newcommand{\rTsb}[1]{\ensuremath{\rT{\lsb#1\rsb}}}

\newcommand{\cT}[1]{\ensuremath{#1^{\ast}}}
\newcommand{\cTb}[1]{\ensuremath{\cT{\lb{}#1\rb}}}

\newcommand{\si}[1]{\ensuremath{#1^{+}}}

\newcommand{\inv}[1]{\ensuremath{{#1}^{\textsf{-1}}}}
\newcommand{\invb}[1]{\ensuremath{\inv{\lb{#1}\rb}}}

\newcommand{\grad}[1][]{\ensuremath{\nabla_{#1}}}

\let\div\relax
\DeclareMathOperator{\div}{div}
\newcommand{\divb}[2][]{\ensuremath{\div_{#1}\lb{#2}\rb}}

\newcommand{\sodeldel}[1]{\ensuremath{\frac{\partial^2}{\partial^2 #1}}} 

\newcommand{\dd}[1]{\ensuremath{\frac{\mathrm{d}}{\mathrm{d}#1}}}

\newcommand{\ddfy}{\ensuremath{\dd{\fy}}}
\newcommand{\ddt}{\ensuremath{\dd{t}}}

\newcommand{\J}{\ensuremath{\mathbb{J}}}
\newcommand{\Jtwo}[1]{\ensuremath{\J_{2#1}}}
\newcommand{\TJtwo}[1]{\ensuremath{\rT{\J}_{2#1}}}
\newcommand{\Jtn}{\ensuremath{\Jtwo{n}}}
\newcommand{\TJtn}{\ensuremath{\TJtwo{n}}}
\newcommand{\Jtk}{\ensuremath{\Jtwo{k}}}
\newcommand{\TJtk}{\ensuremath{\TJtwo{k}}}

\newcommand{\I}[1]{\ensuremath{\fI_{#1}}}

\newcommand{\Z}[1]{\ensuremath{\fzero}_{#1}}

\newcommand{\norm}[1]{\left\lVert{#1}\right\rVert}
\newcommand{\wnorm}[2]{\ensuremath{\norm{#2}_{#1}}} 
\newcommand{\Fnorm}[1]{\wnorm{\mathrm{F}}{#1}}
\newcommand{\tnorm}[1]{\wnorm{\rm{2}}{#1}} 

\newcommand{\ip}[2]{\left\langle #1,\; #2 \right\rangle}

\newcommand{\tInit}{\ensuremath{t_{\mathrm{0}}}}
\newcommand{\tEnd}{\ensuremath{t_{\mathrm{end}}}}
\newcommand{\ftInterval}{\ensuremath{[\tInit, \tEnd]}} 
\newcommand{\fxInit}{\ensuremath{\fx_{\mathrm{0}}}}

\newcommand{\paramSet}{\mathcal{P}}

\newcommand{\ns}{\ensuremath{n_{\mathrm{s}}}}
\newcommand{\xs}{\ensuremath{\fx^{\mathrm{s}}}}

\newcommand{\Xs}{\ensuremath{\fX_{\mathrm{s}}}}

\newcommand{\reduce}[1]{\ensuremath{#1_{\mathrm{r}}}}
\newcommand{\fxr}{\ensuremath{\reduce{\fx}}}
\newcommand{\reconstruced}[1]{\ensuremath{{#1}_{\mathrm{rc}}}}
\newcommand{\fxrc}{\ensuremath{\reconstruced{\fx}}}
\newcommand{\redSpace}{\ensuremath{\mathcal{V}}}
\newcommand{\fAr}{\ensuremath{\reduce{\fA}}}
\newcommand{\fbr}{\ensuremath{\reduce{\fb}}}

\newcommand{\solManifold}{\ensuremath{\mathcal{S}}}
\newcommand{\aprxSolManifold}{\ensuremath{\widehat{\solManifold}_{\fV\fW}}}

\newcommand{\gradx}{\ensuremath{\nabla_{\fx}}}

\newcommand{\symplForm}[1][]{\ensuremath{\omega_{#1}}}
\newcommand{\symplFormb}[3][]{\ensuremath{\symplForm[#1]\lb #2,\; #3 \rb}}

\newcommand{\Jtns}{\ensuremath{\Jtwo{\ns}}}
\newcommand{\TJtns}{\ensuremath{\TJtwo{\ns}}}

\newcommand{\Ham}{\ensuremath{\mathcal{H}}}

\newcommand{\XHam}{\ensuremath{\fX_{\Ham}}}

\newcommand{\Hamr}{\ensuremath{\reduce{\Ham}}}
\newcommand{\gradxr}{\ensuremath{\nabla_{\fxr}}}

\newcommand{\fHr}{\ensuremath{\reduce{\fH}}}
\newcommand{\fhr}{\ensuremath{\reduce{\fh}}}

\newcommand{\Ys}{\ensuremath{\fY_{\mathrm{s}}}}

\newcommand{\fVE}{\ensuremath{\fV_{\fE}}}

\newcommand{\extended}[1]{#1^{\mathrm{e}}}
\newcommand{\Hame}{\ensuremath{\extended{\Ham}}}
\newcommand{\fxe}{\ensuremath{\extended{\fx}}}
\newcommand{\qe}{\ensuremath{\extended{q}}}
\newcommand{\pe}{\ensuremath{\extended{p}}}

\newcommand{\fxer}{\ensuremath{\reduce{\fxe}}}
\newcommand{\Hamer}{\ensuremath{\reduce{\Hame}}}

\newcommand{\PSD}{\texttt{PSD}}
\newcommand{\Vspace}{\ensuremath{\mathbb{V}}}
\newcommand{\symplDomain}{\ensuremath{U}}

\newcommand{\flow}[1][t]{\ensuremath{\fphi_{#1}}}

\newcommand{\errHam}{\ensuremath{e_{\Ham}}}

\newcommand{\fPVE}{\ensuremath{\fP_{\fVE}}}
\newcommand{\fPVEoc}{\ensuremath{\fPVE^{\perp}}}

\newcommand{\Cs}{\ensuremath{\fC_{\textrm{s}}}}
\newcommand{\fqs}{\ensuremath{\fq^{\textrm{s}}}}
\newcommand{\fps}{\ensuremath{\fp^{\textrm{s}}}}
\newcommand{\fUC}{\ensuremath{\fU_{\Cs}}}

\newcommand{\rhoz}{\ensuremath{\rho_{0}}}
\newcommand{\lamelambda}{\ensuremath{\lambda_{\textrm{L}}}}
\newcommand{\lamemu}{\ensuremath{\mu_{\textrm{L}}}}
\newcommand{\domain}{\ensuremath{\varOmega}}
\newcommand{\boundary}{\ensuremath{\varGamma}}
\newcommand{\boundaryDirichlet}{\ensuremath{\boundary_{\fu}}}
\newcommand{\boundaryNeumann}{\ensuremath{\varGamma_{\ft}}}

\newcommand{\nondim}[1]{#1^{\textrm{c}}}

\newcommand{\nt}{\ensuremath{n_{\textrm{t}}}}

\newcommand{\err}{\ensuremath{e}}

\newcommand{\meanErr}{\ensuremath{\overline{\err}}}

\newcommand{\orthoMeasure}{\ensuremath{o_{\fV}}}
\newcommand{\symMeasure}{\ensuremath{s_{\fV}}}

\newcommand{\errProj}{\ensuremath{e_{l_{2}}}}

\newcommand{\fVPOD}{\ensuremath{\fV_{\textrm{POD}}}}

\newcommand{\CT}[1]{{#1}_{\textrm{CT}}}

\newcommand{\cSVD}[1]{{#1}_{\textrm{cSVD}}}

\newcommand{\fSigmaS}{\ensuremath{\fSigma_{\mathrm{s}}}}
\newcommand{\sigmaS}{\ensuremath{\sigma^{\mathrm{s}}}}
\newcommand{\idxSetPSD}{\ensuremath{\mathcal{I}_{\textrm{PSD}}}}
\newcommand{\idxSetPSDtk}{\ensuremath{\idxSetPSD^{2k}}}
\newcommand{\wsigmaS}{\ensuremath{w^\mathrm{s}}}

\newcommand{\prcCSVD}{\texttt{cSVD}}
\newcommand{\prcSVD}{\texttt{SVD}}

\newcommand{\usim}{\ensuremath{\textrm{m}}} 
\newcommand{\usis}{\ensuremath{\textrm{s}}} 
\newcommand{\usiN}{\ensuremath{\textrm{N}}} 
\newcommand{\usikg}{\ensuremath{\textrm{kg}}} 

\newcommand{\fRs}{\ensuremath{\widetilde{\fR}}}

\newcommand{\HamRel}{\ensuremath{\Ham_{\mathrm{rel}}}}

\newcommand{\infnorm}[1]{\ensuremath{\norm{#1}_{\infty}}}

\usepackage{import}
\usepackage{color}

\usepackage{pgfplots}
\usepgfplotslibrary{patchplots}
\usetikzlibrary{calc}
\usetikzlibrary{matrix}

\definecolor{uniSblue}{RGB}{0,65,145}
\definecolor{uniSlightblue}{RGB}{0,190,255}
\definecolor{uniSgray}{RGB}{62, 68, 76}
\definecolor{uniSyellow}{RGB}{236, 178, 32}
\pgfplotsset{%
  colormap={unis}{color=(uniSblue) color=(uniSlightblue)},%
  colormap={unisgray}{color=(uniSgray) color=(uniSblue)}%
}%

\usepgfplotslibrary{groupplots}

\definecolor{uniSred}{RGB}{237, 28, 36}
\definecolor{uniSgreen}{RGB}{0, 166, 81}
\colorlet{uniSlightgray}{uniSgray!30}

\colorlet{color1}{uniSgray}
\colorlet{color2}{uniSblue}
\colorlet{color3}{uniSlightblue}
\colorlet{color4}{uniSyellow}
\colorlet{color5}{uniSred}
\colorlet{color6}{uniSgreen}

\pgfplotsset{cycle list = {%
	color1,%
	color2,%
	color3,%
	color4,%
	color5,%
	color6,%
	},
}

\tikzset{%
    draw=uniSgray
}
\pgfplotsset{
  every axis/.append style={
	  draw=uniSgray,
  },
}

\pgfplotsset{
  every axis plot/.append style={line width=1pt,},
  every axis/.append style={
    ymajorgrids,
    xmajorgrids,
    grid style={dashed, lightgray,semithick},
   	axis line style = semithick,
   	every tick/.style={semithick,},
  },
}

\pgfplotsset{
  /pgfplots/group/every plot/.append style = {
    height = 5cm, width = 5cm,
    ylabel near ticks,
  },
  /pgfplots/group/horizontal sep = {2.5cm},
}

\pgfplotsset{
    legend image with text/.style={
        legend image code/.code={%
            \node[anchor=center] at (0.3cm,0cm) {#1};
        }
    },
}

\pgfplotsset{
	/pgfplots/contour legend/.style={
		legend image code/.code={
			\draw[color={uniSlightblue}, line width=0.8pt] (0.3cm,0.05cm) circle (0.2cm);
			\draw[color={uniSlightblue!50!uniSblue}, line width=0.8pt] (0.3cm,0.05cm) circle (0.125cm);
			\draw[color={uniSblue}, line width=0.8pt] (0.3cm,0.05cm) circle (0.05cm);
		},
	},
}

\newcommand{\save}[3]{\expandafter\edef\csname #1#2\endcsname{#3}}
\newcommand{\load}[2]{\expandafter\csname #1#2\endcsname}

\pgfplotstableread[col sep=semicolon,trim cells]{
name                    ; alias             ; mark     ; color
RB_size                 ; {}                ; {}       ; {}
POD_full_state          ; POD full          ; o        ; color1
POD_separate_states     ; POD sep.          ; x        ; color2
PSD_complex_SVD         ; PSD cSVD          ; triangle ; color3
PSD_greedy              ; PSD greedy        ; asterisk ; color4
PSD_SVD_like_decomp     ; PSD SVD-like      ; diamond  ; color5
}\tablenames

\begin{document}

\maketitle

\begin{abstract}
Parametric high-fidelity simulations are of interest for a wide range of applications. But the restriction of computational resources renders such models to be inapplicable in a real-time context or in multi-query scenarios. Model order reduction (MOR) is used to tackle this issue. Recently, MOR is extended to preserve specific structures of the model throughout the reduction, e.g.\ structure-preserving MOR for Hamiltonian systems. This is referred to as symplectic MOR. It is based on the classical projection-based MOR and uses a symplectic reduced order basis (ROB). Such a ROB can be derived in a data-driven manner with the Proper Symplectic Decomposition~(PSD) in the form of a minimization problem. Due to the strong nonlinearity of the minimization problem, it is unclear how to efficiently find a global optimum. In our paper, we show that current solution procedures almost exclusively yield suboptimal solutions by restricting to orthonormal ROBs. As new methodological contribution, we propose a new method which eliminates this restriction by generating non-orthonormal ROBs. In the numerical experiments, we examine the different techniques for a classical linear elasticity problem and observe that the non-orthonormal technique proposed in this paper shows superior results with respect to the error introduced by the reduction.
\end{abstract}\hspace{10pt}


\section{Introduction}

Simulations enable researchers of all fields to run virtual experiments that are too expensive or impossible to be carried out in the real world. In many contexts, high-fidelity models are indispensable to represent the simulated process accurately. These high-fidelity simulations typically come with the burden of large computational cost such that an application in real-time or an evaluation for many different parameters is impossible respecting the given restrictions of computational resources at hand. Model order reduction (MOR) techniques can be used to reduce the computational cost of evaluations of the high-fidelity model by approximating these with a surrogate reduced-order model (ROM) \cite{LuminyBook2017}.

One class of high-fidelity models are systems of ordinary differential equations (ODEs) with a high order, i.e.\ a high dimension in the unknown variable. Such models typically arise from fine discretizations of time-dependent partial differential equations (PDEs). Since each point in the discretization requires one or multiple unknowns, fine discretizations with many discretization points yield a system of ODEs with a high order. In some cases, the ODE system takes the form of a finite-dimensional Hamiltonian system. Examples are linear elastic models \cite{Buchfink2018} or gyro systems \cite{Xu2005}.

Symplectic MOR \cite{Peng2016} allows to derive a ROM for high-dimensional Hamiltonian systems by lowering the order of the system while maintaining the Hamiltonian structure. Thus, it is also referred to as structure-preserving MOR for Hamiltonian systems \cite{Maboudi2017}. Technically speaking, a Petrov--Galerkin projection is used in combination with a symplectic reduced-order basis (ROB).

For a data-driven generation of the ROB, the conventional methods e.g.\ the Proper Orthogonal Decomposition (POD) \cite{LuminyBook2017} are not suited since they do not necessarily compute a symplectic ROB. To this end, the referenced works introduce the Proper Symplectic Decomposition (PSD) which is a data-driven basis generation technique for symplectic ROBs. Due to the high nonlineariy of the optimization problem, an efficient solution strategy is yet unknown for the PSD. The existing PSD methods (Cotangent Lift, Complex SVD, a nonlinear programming approach \cite{Peng2016} and a greedy procedure introduced in \cite{Maboudi2017}) each restrict to a specific subset of symplectic ROBs from which they select optimal solutions which might be globally suboptimal.

The present paper classifies the existing symplectic basis generation techniques in two classes of methods which either generate orthonormal or non-orthonormal bases. To this end, we show that the existing basis generation techniques for symplectic bases almost exclusively restrict to orthonormal bases. Furthermore, we prove that Complex SVD is the optimal solution of the PSD on the set of orthonormal, symplectic bases. During the proof, an alternative formulation of the Complex SVD for symplectic matrices is introduced. To leave the class of orthonormal, symplectic bases, we propose a new basis generation technique, namely the PSD SVD-like decomposition. It is based on an SVD-like decomposition of arbitrary matrices $\fB \in \R^{n \times 2m}$ introduced in \cite{Xu2003}.

This paper is organized in the following way: \Cref{sec:MORAutoHamSys} is devoted to the structure-preserving MOR for autonomous and non-autonomous, parametric Hamiltonian systems and thus, introduces symplectic geometry, Hamiltonian systems and symplectic MOR successively. The data-driven generation of a symplectic ROB with PSD is discussed in \Cref{subsec:PSD}. The numerical results are presented and elaborated in \Cref{sec:NumRes} exemplified by a Lam\'e--Navier type elasticity model which we introduce at the beginning of that section together with a short comment on the software that is used for the experiments. The paper is summarized and concluded in \cref{sec:Conclusion}.

\section{Symplectic model reduction}
\label{sec:MORAutoHamSys}
Symplectic MOR for autonomous Hamiltonian systems is introduced in \cite{Peng2016}. We repeat the essentials for the sake of completeness and to provide a deeper understanding of the methods used. In the following $\fmu \in \paramSet \subset \R^p$ describe $p \in \N$ parameters of the system from the parameter set $\paramSet$. We might skip the explicit dependence on the parameter vector $\fmu$ if it is not relevant in this specific context.

\subsection{Symplectic geometry in finite dimensions}
\label{subsec:SymplGeo}
\begin{Definition}[Symplectic form over $\R$]
  Let $\Vspace$ be a finite-dimensional vector space over $\R$. We consider a skew-symmetric and non-degenerate bilinear form $\symplForm: \Vspace \times \Vspace \rightarrow \R$ , i.e.\ for all $\fv_1,\fv_2 \in \Vspace$, it holds
\begin{align*}
	&\symplFormb{\fv_1}{\fv_2} = -\symplFormb{\fv_2}{\fv_1}&
&\quad\text{and}\quad&
	&\symplFormb{\fv_2}{\fv_3} = 0 \quad \forall \fv_3 \in \Vspace
\implies \fv_3 = \fzero.
\end{align*}
The bilinear form $\symplForm$ is called symplectic form on $\Vspace$ and the pair $(\Vspace, \symplForm)$ is called symplectic vector space.
\end{Definition}

It can be shown that $\Vspace$ is necessarily of even dimension \cite{daSilva2008}. Thus, $\Vspace$ is isomorphic to $\R^{2n}$ which is why we restrict to $\Vspace = \R^{2n}$ and write $\symplForm[2n]$ instead of $\symplForm$ in the following. In context of the theory of Hamiltonians, $\R^{2n}$ refers to the phase space which consists, in the context of classical mechanics, of position states $\fq = \rTsb{q_1, \dots, q_n} \in \R^n$ of the configuration space and momentum states $\fp = \rTsb{p_1, \dots, p_n} \in \R^n$ which form together the state $\fx = \rTsb{q_1, \dots, q_n, p_1, \dots, p_n} \in \R^{2n}$.

It is guaranteed \cite{daSilva2008} that there exists a basis $\lcb \fe_1, \dots, \fe_n, \ff_1, \dots, \ff_n \rcb \subset \R^{2n}$ such that the symplectic form takes the canonical structure
\begin{align} \label{eq:CanonicalSymplForm}
	&\symplFormb[2n]{\fv_1}{\fv_2} = \rT\fv_1 \Jtn \fv_2
\quad \forall \fv_1, \fv_2 \in \R^{2n}, \qquad&
&\Jtn := 
\begin{bmatrix}
	\Z{n}  & \I{n} \\
	-\I{n} & \Z{n}
\end{bmatrix},
\end{align}
where $\I{n} \in \R^{n \times n}$ is the identity matrix, $\Z{n} \in \R^{n \times n}$ is the matrix of all zeros and $\Jtn$ is called Poisson matrix. Thus, we restrict to symplectic forms of the canonical structure in the following. For the Poisson matrix, it holds for any $\fv \in \R^{2n}$
\begin{align}\label{eq:StructMat}
	&\Jtn \TJtn = \I{2n},\quad&
&\Jtn \Jtn = \TJtn \TJtn = -\I{2n},\quad&
&\rT\fv \Jtn \fv = 0.
\end{align}
These properties are intuitively understandable as the Poisson matrix is a $2n$-dimensional, $90^\circ$ rotation matrix and the matrix $-\I{2n}$ can be interpreted as a rotation by $180^{\circ}$ in this context.

\begin{Definition}[Symplectic map]
  Let $A: \R^{2m} \rightarrow \R^{2n}$, $\fy \mapsto \fA \fy$, $\fA \in \R^{2n \times 2m}$ be a linear mapping for $n,m \in \N$ and $m\leq n$. We call $A$ a linear symplectic map and $\fA$ a symplectic matrix with respect to $\symplForm[2n]$ and $\symplForm[2m]$ if
  \begin{align} \label{eq:SymplMat}
    \rT\fA \Jtn \fA = \Jtwo{m}.
  \end{align}
  where $\symplForm[2m]$ is the canonical symplectic form on $\R^{2m}$ (and is equal to $\symplForm[2n]$ if $n=m$).
  
  Let $\symplDomain \subset \R^{2m}$ be an open set and $\fg:\symplDomain \rightarrow \R^{2n}$ a differentiable map on $\symplDomain$. We call $\fg$ a symplectic map if the Jacobian matrix $\ddfy\fg(\fy) \in \R^{2n \times 2m}$ is a symplectic matrix for every $\fy \in \symplDomain$.
\end{Definition}

For a linear map, it is easy to check that the condition \cref{eq:SymplMat} is equivalent to the preservation of the symplectic form, i.e.\ for all $\fv_1, \fv_2 \in \R^{2m}$
\begin{align*}
  \symplFormb[2n]{ \fA \fv_1 }{ \fA \fv_2 }
= \rT\fv_1 \rT\fA \Jtn \fA \fv_2
= \rT\fv_1 \Jtwo{m} \fv_2
= \symplFormb[2m]{\fv_1}{\fv_2}.
\end{align*}

Now we give the definition of the so-called symplectic inverse which will be used in \Cref{subsec:SymplMOR}.

\begin{Definition}[Symplectic inverse]
	For each symplectic matrix $\fA \in \R^{2n \times 2m}$, we define the symplectic inverse
	\begin{align}\label{eq:SymplInv}
		\si\fA = \TJtwo{m} \rT\fA \Jtn \in \R^{2m \times 2n}.
	\end{align}
\end{Definition}
The symplectic inverse $\si\fA$ exists for every symplectic matrix and it holds the following inverse relation
\begin{align*}
  \si{\fA} \fA = \TJtwo{m} \rT\fA \Jtn \fA = \TJtwo{m} \Jtwo{m} = \I{2m}.
\end{align*}

\subsection{Finite-dimensional, autonomous Hamiltonian systems}
\label{subsec:AutoHam}
To begin with, we introduce the Hamiltonian system in a finite-dimensional, autonomous setting.

\begin{Definition}[Finite-dimensional, autonomous Hamiltonian system]
  Let $\Ham:\R^{2n} \times \paramSet \rightarrow \R$ be a scalar-valued function that we require to be continuously differentiable in the first argument and which we call Hamiltonian (function). Hamilton's equation is an initial value problem with the prescribed initial data $\tInit \in \R$, $\fxInit(\fmu) \in \R^{2n}$ which describes the evolution of the solution $\fx(t, \fmu) \in \R^{2n}$ for all $t \in \ftInterval, \; \fmu \in \paramSet$ with
  \begin{align}\label{eq:AutoHamEq}
		\ddt \fx(t, \fmu) =&\; \Jtn \gradx \Ham(\fx(t, \fmu), \fmu) =: \XHam(\fx(t, \fmu), \fmu), \quad&
		\fx(\tInit, \fmu) =&\; \fxInit(\fmu)
  \end{align}
  where $\XHam(\bullet, \fmu)$ is called Hamiltonian vector field. The triple $(\Vspace, \symplForm[2n], \Ham)$ is referred to as Hamiltonian system. We denote the flow of a Hamiltonian system as the mapping $\flow: \R^{2n} \times \paramSet \rightarrow \R^{2n}$ that evolves the initial state $\fxInit(\fmu) \in \R^{2n}$ to the corresponding solution $\fx(t, \fmu;\; \tInit, \fxInit(\fmu))$ of Hamilton's equation
\begin{align*}
	\flow(\fxInit, \fmu) := \fx(t, \fmu;\; \tInit, \fxInit(\fmu)),
\end{align*}
where $\fx(t, \fmu;\; \tInit, \fxInit(\fmu))$ indicates that it is the solution with the initial data $\tInit,\fxInit(\fmu)$.
\end{Definition}

The two characteristic properties of Hamiltonian systems are (a) the preservation of the Hamiltonian function and (b) the symplecticity of the flow.

\begin{Proposition}[Preservation of the Hamiltonian]
	The flow of Hamilton's equation $\flow$ preserves the Hamiltonian function $\Ham$.
\end{Proposition}

\begin{proof}
	We prove the assertion by showing that the evolution over time is constant for any $\fx \in \R^{2n}$ due to
	\begin{align*}
		\ddt \Ham(\flow(\fx))
= \rTb{\gradx\Ham(\flow(\fx))} \ddt \flow(\fx)
\stackrel{\cref{eq:AutoHamEq}}{=} \rTb{\gradx\Ham(\flow(\fx))} \Jtn \gradx\Ham(\flow(\fx))
\stackrel{\cref{eq:StructMat}}{=} 0.
	\end{align*}
\end{proof}

\begin{Proposition}[Symplecticity of the flow]
	Let the Hamiltonian function be twice continuously differentiable in the first argument. Then, the flow $\flow(\bullet, \fmu): \R^{2n} \rightarrow \R^{2n}$ of a Hamiltonian system is a symplectic map.
\end{Proposition}

\begin{proof}
	See \cite[Chapter VI, Theorem 2.4]{Hairer2006}.
\end{proof}

\subsection{Symplectic model order reduction for autonomous Hamiltonian systems}
\label{subsec:SymplMOR}
The goal of MOR \cite{LuminyBook2017} is to reduce the order, i.e.\ the dimension, of high dimensional systems. To this end, we approximate the high-dimensional state $\fx(t) \in \R^{2n}$ with
\begin{align*}
	\fx(t, \fmu) \approx \fxrc(t, \fmu) = \fV \fxr(t, \fmu),\quad&
&\redSpace = \colspanb{\fV}
\end{align*}
with the reduced state $\fxr(t) \in \R^{2k}$, the reduced-order basis (ROB) $\fV \in \R^{2n \times 2k}$, the reconstructed state $\fxrc(t) \in \redSpace$ and the reduced space $\redSpace \subset \R^{2n}$. The restriction to even-dimensional spaces $\R^{2n}$ and $\R^{2k}$ is not necessary for MOR in general but is required for the symplectic MOR in the following. To achieve a computational advantage with MOR, the approximation should introduce a clear reduction of the order, i.e.\ $2k \ll 2n$.

For Petrov--Galerkin projection-based MOR techniques, the ROB $\fV$ is accompanied by a projection matrix $\fW \in \R^{2n \times 2k}$ which is chosen to be biorthogonal to $\fV$, i.e. $\rT\fW \fV = \I{2k}$. The reduced-order model (ROM) is derived with the requirement that the residual $\fr(t, \fmu)$ vanishes in the space spanned by the columns of the projection matrix, i.e.\ in our case
\begin{align}\label{eq:Residual}
	&\fr(t, \fmu) = \ddt\fxrc(t, \fmu) - \XHam(\fxrc(t, \fmu), \fmu) \in \R^{2n},\quad&
&\rT\fW \fr(t, \fmu) = \Z{2k \times 1},
\end{align}
where $\Z{2k \times 1} \in \R^{2k}$ is the vector of all zeros. Due to the biorthogonality, this is equivalent to
\begin{align}\label{eq:ROM}
	&\ddt \fxr(t, \fmu) = \rT\fW \XHam(\fxrc(t, \fmu), \fmu) = \rT\fW \Jtn \gradx \Ham(\fxrc(t, \fmu), \fmu),\quad&
&\fxr(\tInit, \fmu) = \rT\fW \fxInit(\fmu).
\end{align}

In the context of symplectic MOR, the ROB is chosen to be a symplectic matrix \cref{eq:SymplMat} which we call a symplectic ROB. Additionally, the transposed projection matrix is the symplectic inverse $\rT\fW =\si\fV$ and the projection in \cref{eq:ROM} is called a symplectic projection or symplectic Galerkin projection \cite{Peng2016}. The (possibly oblique) projection reads
\begin{align*}
  \fP = \fV \invb{\rT\fW \fV} \rT\fW = \fV \invb{\si\fV \fV} \si\fV = \fV \si\fV.
\end{align*}
In combination, this choice of $\fV$ and $\fW$ guarantees that the Hamiltonian structure is preserved by the reduction which is shown in the following proposition.

\begin{Proposition}[Reduced autonomous Hamiltonian system]\label{theo:MORAutoHamSys}
	Let $\fV$ be a symplectic ROB with the projection matrix $\rT\fW = \si\fV$. Then, the ROM \cref{eq:ROM} of a high-dimensional Hamiltonian system $(\R^{2n}, \symplForm[2n], \Ham)$ is a Hamiltonian system $(\R^{2k}, \symplForm[2k], \Hamr)$ on $\R^{2k}$ with the canonical symplectic form $\symplForm[2k]$ and the reduced Hamiltonian function $\Hamr(\fxr, \fmu) = \Ham(\fV \fxr, \fmu)$ for all $\fxr \in \R^{2k}$.
\end{Proposition}

\begin{proof}
	First, we remark that the symplectic inverse is a valid biorthogonal projection matrix since it fulfils $\rT\fW \fV = \si\fV \fV = \I{2k}$. To derive the Hamiltonian form of the ROM in \cref{eq:ROM}, we use the identity
	\begin{align}\label{eq:RelationWV}
		\rT\fW \Jtn
= \si\fV \Jtn
\stackrel{\eqref{eq:SymplInv}}{=} \TJtk \rT\fV \Jtn \Jtn
= - \TJtk \rT\fV
= \Jtk \rT\fV,
	\end{align}
	which makes use of the properties \cref{eq:StructMat} of the Poisson matrix. It follows with \cref{eq:AutoHamEq,eq:ROM,eq:RelationWV}
	\begin{align*}
		\ddt \fxr(t)
= \rT\fW \Jtn \gradx \Ham(\fxrc(t))
= \Jtk \rT\fV \gradx \Ham(\fxrc(t))
= \Jtk \gradxr \Hamr(\fxr(t))
	\end{align*}
	where the last step follows from the chain rule. Thus, the evolution of the reduced state takes the form of Hamilton's equation and the resultant ROM is equal to the Hamiltonian system $(\R^{2k}, \symplForm[2k], \Hamr)$.
\end{proof}

\begin{Corollary}[Linear Hamiltonian system]\label{cor:QuadHam}
  Hamilton's equation is a linear system in the case of a quadratic Hamiltonian $\Ham(\fx, \fmu) = \txtfrac{1}{2} \; \rT\fx \fH(\fmu) \fx + \rT\fx \fh(\fmu)$ with $\fH(\fmu) \in \R^{2n \times 2n}$ symmetric and $\fh(\fmu) \in \R^{2n}$
	\begin{align}\label{eq:LinSys}
		\ddt \fx(t, \fmu) = \fA(\fmu) \fx(t,\fmu) + \fb(\fmu),\quad&
&\fA(\fmu) = \Jtn \fH(\fmu),\quad&
&\fb(\fmu) = \Jtn \fh(\fmu).
	\end{align}
	The evolution of the reduced Hamiltonian system reads
	\begin{align*}
		\ddt \fxr(t, \fmu) = \fAr(\fmu) \fxr(t,\fmu) + \fbr(\fmu),\quad&
&\begin{split}
	\fAr(\fmu) =&\; \Jtk \fHr(\fmu) \stackrel{\cref{eq:RelationWV}}{=} \rT\fW \fA(\fmu) \fV,\\
	\fbr(\fmu) =&\; \Jtk \fhr(\fmu) \stackrel{\cref{eq:RelationWV}}{=} \rT\fW \fb(\fmu) \fV,
\end{split}&
&\begin{split}
	\fHr(\fmu) =&\; \rT\fV \fH(\fmu) \fV,\\
	\fhr(\fmu) =&\; \rT\fV \fh(\fmu).
\end{split}
	\end{align*}
	with the reduced Hamiltonian function $\Hamr(\fxr, \fmu) = \txtfrac{1}{2}\; \rT\fxr \fHr(\fmu) \fxr + \rT\fxr \fhr(\fmu)$.
\end{Corollary}

\begin{Remark}
We emphasise that the reduction of linear Hamiltonian systems follows the pattern of the classical projection-based MOR approaches \cite{Haasdonk2011b} to derive the reduced model with $\fAr = \rT\fW \fA \fV$ and $\fbr = \rT\fW \fb$ which allows a straightforward implementation in existing frameworks.
\end{Remark}

Since the ROM is a Hamiltonian system, it preserves its Hamiltonian. Thus, it can be shown that the error in the Hamiltonian $\errHam(t, \fmu) = \Ham(\fx(t, \fmu), \fmu) - \Hamr(\fxr(t, \fmu), \fmu)$ is constant \cite{Peng2016}. Furthermore, there are a couple of results for the preservation of stability \cite[Theorem 18]{Maboudi2017}, \cite[Section 3.4.]{Peng2016} under certain assumptions on the Hamiltonian function.
%
%
%
%
%

\begin{Remark}[Offline/online decomposition]
A central concept in the field of MOR for parametric systems is the so-called offline/online decomposition. The idea is to split the procedure in a possibly costly offline phase and a cheap online phase where the terms costly and cheap refer to the computational cost. In the offline phase, the ROM is constructed. The online phase is supposed to evaluate the ROM fast. The ultimate goal is to avoid any computations that depend on the high dimension $2n$ in the online phase.

For a linear system, the offline/online decomposition can be achieved if $\fA(\fmu)$, $\fb(\fmu)$ and $\fxInit(\fmu)$ allow a parameter-separability condition \cite{Haasdonk2011b}.

For systems with non-linear parts, multiple approaches \cite{Barrault2004,Chaturantabut2009} exist to enable an offline/online decomposition by introducing an approximation of the non-linear terms. This allows an online-efficient MOR of non-linear systems. For symplectic MOR, the symplectic discrete empirical interpolation method (SDEIM) was introduced \cite[Section 5.2.]{Peng2016} to preserve the symplectic structure throughout the approximation of the non-linear terms.
\end{Remark}

\subsection{Finite-dimensional, non-autonomous Hamiltonian systems}
\label{subsec:NonAutoHam}
Non-autonomous Hamiltonian systems can be redirected to the case of autonomous systems if differentiability with respect to the time is assumed for the Hamiltonian function. The concept of the extended phase space is used. We briefly introduce the approach and explain the link to the symplectic MOR.

\begin{Definition}[Finite-dimensional, non-autonomous Hamiltonian system]
  Let $\Ham:\R \times \R^{2n} \times \paramSet \rightarrow \R$ be a scalar-valued function function that is continuously differentiable in the second argument. A non-autonomous (or time-dependent) Hamiltonian system $(\R^{2n}, \symplForm[2n], \Ham)$ is of the form
  \begin{align}\label{eq:NonAutoHamEq}
    \fx(t, \fmu) = \Jtn \gradx \Ham(t, \fx(t, \fmu), \fmu).
  \end{align}
  We therefore call $\Ham(t, \fx)$ a time-dependent Hamiltonian function.
\end{Definition}

A problem for non-autonomous Hamiltonian systems occurs as the explicit time dependence of the Hamiltonian function introduces an additional variable, the time, and the carrier manifold becomes odd-dimensional. As mentioned in \cref{subsec:SymplGeo}, symplectic vector spaces are always even-dimensional which is why a symplectic description is no longer possible. Different approaches are available to circumvent this issue.
%

As suggested in \cite[Section 4.3]{Maboudi2018}, we use the methodology of the so-called symplectic extended phase space \cite[Chap.\ VI, Sec.\ 10]{Lanczos1940} to redirect the non-autonomous system to an autonomous system. The formulation is based on the extended Hamiltonian function $\Hame: \R^{2n+2} \rightarrow \R$ with
\begin{align}\label{eq:ExtHam}
  &\Hame(\fxe) = \Ham(\qe, \fx) + \pe,&
  &\fxe = \rTb{\qe\; \fq\; \pe\; \fp} \in \R^{2n + 2},&
  &\fx = \rTb{\fq\; \fp} \in \R^{2n},&
  &\qe,\pe \in \R.
\end{align}
Technically, the time is added to the extended state $\fxe$ with $\qe = t$ and the corresponding momentum $\pe = -\Ham(t, \fx(t))$ is chosen such that the extended system is an autonomous Hamiltonian system.

This procedure requires the time-dependent Hamiltonian function to be differentiable in the time variable. Thus, it does for example not allow for the description of loads that are not differentiable in time in the context of mechanical systems. This might, e.g., exclude systems that model mechanical contact since loads that are not differentiable in time are required.

\subsection{Symplectic model order reduction of non-autonomous Hamiltonian systems}
\label{subsec:SymplMORExtSys}

For the MOR of the, now autonomous, extended system, only the original phase space variable $\fx \in \R^{2n}$ is reduced. The time and the corresponding conjugate momentum $\qe,\pe$ are not reduced. To preserve the Hamiltonian structure, a symplectic ROB $\fV \in \R^{2n \times 2k}$ is used for the reduction of $\fx \in \R^{2n}$ analogous to the autonomous case. The result is a reduced extended system which again can be written as a non-autonomous Hamiltonian system $(\R^{2k}, \symplForm[2k], \Hamr)$ with the time-dependent Hamiltonian $\Hamr(t,\fxr,\fmu) = \Ham(t,\fV \fxr, \fmu)$ for all $(t, \fxr) \in [\tInit, \tEnd] \times \R^{2k}$.

An unpleasant side effect of the extended formulation is that the linear dependency on the additional state variable $\pe$ (see \cref{eq:ExtHam}) implies that the Hamiltonian cannot have strict extrema. Thus, the stability results listed in \cite{Peng2016} and \cite{Maboudi2017} do not apply if there is a true time-dependence in the Hamiltonian $\Ham(t, \fx)$. Nevertheless, symplectic MOR in combination with a non-autonomous Hamiltonian system shows stable results in the numerical experiments.

Furthermore, it is important to note that only the extended Hamiltonian $\Hame$ is preserved throughout the reduction. The time-dependent Hamiltonian $\Ham(\cdot, t)$ is not necessarily preserved throughout the reduction, i.e. $\Hame(\fxe(t)) = \Hamer(\fxer(t))$ but potentially $\Ham(\fx(t),t) \neq \Hamr(\fx(t),t)$.

\section{Symplectic basis generation with the Proper Symplectic Decomposition (PSD)}
\label{subsec:PSD}
We yet require a symplectic ROB for symplectic MOR. In the following, we pursue the approach of a ROB generated from a set of snapshots of the system. A snapshot is an element of the so-called solution manifold $\solManifold$ that is approximated with a low-dimensional surrogate $\aprxSolManifold$
\begin{align*}
	&\solManifold := \lcb \fx(t, \fmu) \,\big|\, t \in \ftInterval,\, \fmu \in \paramSet \rcb
\subset \R^{2n},&
&\aprxSolManifold := \lcb \fV \fxr(t, \fmu) \,\big|\, t \in \ftInterval,\, \fmu \in \paramSet \rcb \approx \solManifold.
\end{align*}
In \cite{Peng2016}, the Proper Symplectic Decomposition (PSD) is proposed as a snapshot-based basis generation technique for symplectic ROBs. The idea is to derive the ROB from a minimization problem which is suggested in analogy to the very well established Proper Orthogonal Decomposition (POD, also Principal Component Analysis) \cite{LuminyBook2017}. 

Classically, the POD chooses the ROB $\fVPOD$ to minimize the sum over squared norms of all $\ns \in \N$ residuals $( \I{2n} - \fVPOD \rT\fVPOD ) \xs_i$ of the orthogonal projection $\fVPOD \rT\fVPOD \xs_i$ of the $1\leq i\leq \ns$ single snapshots $\xs_i \in \solManifold$ measured in the 2-norm $\tnorm{\bullet}$ with the constraint that the ROB $\fVPOD$ is orthogonal, i.e.
\begin{align}\label{eq:POD}
	&\minimize{\fVPOD \in \R^{2n \times 2k}}
\sum_{i=1}^{\ns} \tnorm{\lb \I{2n} - \fVPOD \rT\fVPOD \rb \xs_i}^2&
&\textrm{subject to} \quad \rT\fVPOD \fVPOD = \I{2k}.
\end{align}
In contrast, the PSD requires the ROB to be symplectic instead of orthogonal which is expressed in the reformulated constraint. Furthermore, the orthogonal projection is replaced by the symplectic projection $\fV \si\fV \xs_i$ which results in
\begin{align}\label{eq:PSDVectorBased}
	&\minimize{\fV \in \R^{2n \times 2k}}
\sum_{i=1}^{\ns} \tnorm{(\I{2n} - \fV \si\fV) \xs_i}^2&
&\textrm{subject to} \quad \rT\fV \Jtn \fV = \Jtk.
\end{align}
We summarize this in a more compact (matrix-based) formulation in the following definition.

\begin{Definition}[Proper Symplectic Decomposition (PSD)]
  Given $\ns$ snapshots $\xs_1, \dots, \xs_{\ns} \in \solManifold$, we denote the snapshot matrix as $\Xs = [\xs_1, \dots, \xs_{\ns}] \in \R^{2n \times \ns}$. Find a symplectic ROB $\fV \in \R^{2n \times 2k}$ which minimizes
	\begin{align} \label{eq:PSD}
		&\minimize{\fV \in \R^{2n \times 2k}}
\Fnorm{(\I{2n} - \fV \si\fV) \Xs}^2&
&\textrm{subject to} \quad \rT\fV \Jtn \fV = \Jtk,
	\end{align}
	We denote the minimization problem \cref{eq:PSD} in the following as $\PSD(\Xs)$, where $\Xs$ is the given snapshot matrix.
\end{Definition}

The constraint in \cref{eq:PSD} ensures that the ROB $\fV$ is symplectic and thus, guarantees the existence of the symplectic inverse $\si\fV$. Furthermore, the matrix-based formulation \cref{eq:PSD} is equivalent to the vector-based formulation presented in \cref{eq:PSDVectorBased} due to the properties of the Frobenius norm $\Fnorm{\bullet}$.

\subsection{Symplectic, orthonormal basis generation}
\label{sec:SymplOrthonBasisGen}

The foremost problem of the PSD is that there is no explicit solution procedure known so far due to the high nonlinearity and possibly multiple local optima. This is an essential difference to the POD as the POD allows to find a global minimum by solving an eigenvalue problem \cite{LuminyBook2017}.

Current solution procedures for the PSD restrict to a certain subset of symplectic matrices and derive an optimal solution for this subset which might be suboptimal in the class of symplectic matrices. In the following, we show that this subclass almost exclusively restricts to symplectic, orthonormal ROBs.

\begin{Definition}[Symplectic, orthonormal ROB]
  We call a ROB $\fV \in \R^{2n \times 2k}$ symplectic, orthonormal (also orthosymplectic, e.g.\ in \cite{Maboudi2017}) if it is symplectic w.r.t.\ $\symplForm[2n]$ and $\symplForm[2k]$ and is orthonormal, i.e.\ the matrix $\fV$ has orthonormal columns
  \begin{align*}
    &\rT\fV \Jtn \fV = \Jtk&
&\text{and}&
&\rT\fV \fV = \I{2k}.
  \end{align*}
\end{Definition}

In the following, we show an alternative characterization of a symplectic and orthonormal ROB. Therefore, we extend the results given e.g.\ in \cite{Paige1981} for square matrices $\fQ \in \R^{2n \times 2n}$ in the following \cref{lem:SymplOrthonROB} to the case of rectangular matrices $\fV \in \R^{2n \times 2k}$.  This was also partially addressed in \cite[Lemma 4.3.]{Peng2016}.

\begin{Proposition}[Characterization of a symplectic matrix with orthonormal columns]\label{lem:SymplOrthonROB}
The following statements are equivalent for any matrix $\fV \in \R^{2n \times 2k}$
\begin{enumerate}[label=(\roman*)]
  \item $\fV$ is symplectic with orthonormal columns,
  \item $\fV$ is of the form
  \begin{align}\label{eq:SymplOrthonROB}
	&\fV = \lsb \fE\quad \TJtn \fE \rsb =: \fVE \in \R^{2n \times 2k},&
&\fE \in \R^{2n \times k},&
&\rT\fE \fE = \I{k},&
&\rT\fE \Jtn \fE = \Z{k},
  \end{align}
  \item $\fV$ is symplectic and it holds $\rT\fV = \si\fV$.
\end{enumerate}
\end{Proposition}

We remark that these matrices are characterized in \cite{Peng2016} to be elements in $\text{Sp}(2k, \R^{2n}) \cap V_k(\R^{2n})$ where $\text{Sp}(2k, \R^{2n})$ is the symplectic Stiefel manifold and $V_k(\R^{2n})$ is the Stiefel manifold.

\begin{proof}
  ``(i) $\implies$ (ii)'': Let $\fV \in \R^{2n \times 2k}$ be a symplectic matrix with orthonormal columns. We rename the columns to $\fV = [\fE \quad \fF]$ with $\fE = [\fe_1, \dots, \fe_k]$ and $\fF = [\ff_1, \dots, \ff_k]$. The symplecticity of the matrix written in terms of $\fE$ and $\fF$ reads
  \begin{align}\label{eq:SymplROBForEF}
		\rT\fV \Jtn \fV
= 
\begin{bmatrix}
	\rT\fE \Jtn \fE & \rT\fE \Jtn \fF\\
	\rT\fF \Jtn \fE & \rT\fF \Jtn \fF
\end{bmatrix}
=
\begin{bmatrix}
 \Z{k} & \I{k}\\
 -\I{k} & \Z{k}
\end{bmatrix}
\iff
\begin{split}
	\rT\fE \Jtn \fE = \rT \fF \Jtn \fF &= \Z{k},\\
	-\rT\fF \Jtn \fE = \rT\fE \Jtn \fF &= \I{k}.
\end{split}
  \end{align}
  Expressed in terms of the columns $\fe_i,\ff_i$ of the matrices $\fE,\, \fF$, this condition reads for any $1\leq i,j\leq k$
  \begin{align*}
    &\rT\fe_i \Jtn \fe_j = 0, &
&\rT\fe_i \Jtn\ff_j = \delta_{ij}, &
&\rT\ff_i \Jtn \fe_j = -\delta_{ij}, &
&\rT\ff_i \Jtn\ff_j = 0,
  \end{align*}
  and the orthonormality of the columns of $\fV$ implies
  \begin{align*}
    &\rT\fe_i \fe_j = \delta_{ij},&
&\rT\ff_i \ff_j = \delta_{ij}.
  \end{align*}
  For a fixed $i \in \{1, \dots, k\}$, it is easy to show with $\TJtn \Jtn = \I{2n}$ that $\Jtn \ff_i$ is of unit length
  \begin{align*}
    1 = \delta_{ii} = \rT\ff_i \ff_i = \rT\ff_i \TJtn \Jtn \ff_i = \tnorm{\Jtn \ff_i}^2.
  \end{align*}
  Thus, $\fe_i$ and $\Jtn \ff_i$ are both unit vectors which fulfill $\rT\fe_i \Jtn \ff_i = \ip{\fe_i}{\Jtn\ff_i}_{\R^{2n}} = 1$. By the Cauchy-Schwarz inequality, it holds $\ip{\fe_i}{\Jtn\ff_i} = \norm{\fe_i} \norm{\Jtn\ff_i}$ if and only if the vectors are parallel. Thus, we infer $\fe_i = \Jtn \ff_i$ which is equivalent to $\ff_i = \TJtn \fe_i$. Since this holds for all $i \in \{1, \dots, k\}$, we conclude that $\fV$ is of the form proposed in \cref{eq:SymplOrthonROB}.

  ``(ii) $\implies$ (iii)'':
  Let $\fV$ be of the form \cref{eq:SymplOrthonROB}. Direct calculation yields
  \begin{align*}
    \rT\fV \Jtn \fV
= \begin{bmatrix} \rT\fE\\ \rT\fE \Jtn \end{bmatrix} \Jtn \begin{bmatrix} \fE & \TJtn\fE \end{bmatrix}
= \begin{bmatrix} \rT\fE \Jtn \fE & \rT\fE\fE\\ -\rT\fE \fE & \rT\fE \Jtn \fE \end{bmatrix}
\stackrel{\eqref{eq:SymplOrthonROB}}{=} \begin{bmatrix} \Z{k} & \I{k}\\ -\I{k} & \Z{k} \end{bmatrix}
= \Jtk
  \end{align*}
  which shows that $\fV$ is symplectic. Thus, the symplectic inverse $\si\fV$ exists. The following calculation shows that it equals the transposed $\rT\fV$
	\begin{align*}
		\si\fV
  = \TJtk \rT\fV \Jtn
	= \TJtk \begin{bmatrix} \rT\fE\\ \rT\fE\Jtn \end{bmatrix} \Jtn
	= \begin{bmatrix} -\rT\fE\Jtn\\ \rT\fE  \end{bmatrix} \Jtn
	= \begin{bmatrix} -\rT\fE\Jtn\Jtn\\ \rT\fE \Jtn  \end{bmatrix}
	= \begin{bmatrix} \rT\fE\\ \rT\fE\Jtn \end{bmatrix}
	= \rT\fV.
	\end{align*}
	
	``(iii) $\implies$ (i)'': Let $\fV$ be symplectic with $\rT \fV = \si\fV$. Then, we know that $\fV$ has orthonormal columns since
	\begin{align*}
		\I{k} = \si\fV \fV = \rT\fV \fV.
	\end{align*}
\end{proof}

\Cref{lem:SymplOrthonROB} essentially limits the symplectic, orthonormal ROB $\fV$ to be of the form \cref{eq:SymplOrthonROB}. Later in the current section, we see how to solve the PSD for ROBs of this type. In \Cref{sec:NonOrthPSD}, we are interested in ridding the ROB $\fV$ of this requirement to explore further solution methods of the PSD.

As mentioned before, the current solution procedures for the PSD almost exclusively restrict to the class of symplectic, orthonormal ROBs introduced in \cref{lem:SymplOrthonROB}. This includes the Cotangent Lift \cite{Peng2016}, the Complex SVD \cite{Peng2016}, partly the non-linear programming algorithm from \cite{Peng2016} and the greedy procedure presented in \cite{Maboudi2017}. We briefly review these approaches in the following proposition.

\begin{Proposition}[Symplectic, orthonormal basis generation]\label{lem:OrtonSymplBasisGen}
  The Cotangent Lift (CT), Complex SVD (cSVD) and the greedy procedure for symplectic basis generation all derive a symplectic and orthonormal ROB. The non-linear programming (NLP) admits a symplectic, orthonormal ROB if the coefficient matrix $\fC$ in \cite[Algorithm 3]{Peng2016} is symplectic and has orthonormal columns, i.e.\ it is of the form $\fC_{\fG} = [\fG \quad \TJtk \fG]$. The methods can be rewritten with $\fVE = [\fE \quad \TJtn\fE]$, where the different formulations of $\fE$ read
\begin{align*}
    &\fE_{\textrm{CT}} = \begin{bmatrix} \CT{\fPhi}\\ \Z{n \times k} \end{bmatrix}&
&\fE_{\textrm{cSVD}} = \begin{bmatrix} \cSVD{\fPhi}\\ \cSVD{\fPsi} \end{bmatrix},&
&\fE_{\textrm{greedy}} = [\fe_1, \dots, \fe_k],&
&\fE_{\textrm{NLP}} = \widetilde{\fVE} \fG
  \end{align*}
  where 
  \begin{enumerate}[label=(\roman*)]
    \item $\CT{\fPhi}, \cSVD{\fPhi}, \cSVD{\fPsi} \in \R^{n \times k}$ are matrices that fulfil
	  \begin{align*}
	    &\rT{\CT{\fPhi}}\CT{\fPhi} = \I{k},&
	  &\rT{\cSVD{\fPhi}} \cSVD{\fPhi} + \rT{\cSVD{\fPsi}} \cSVD{\fPsi} = \I{k},&
	  &\rT{\cSVD{\fPhi}} \cSVD{\fPsi} = \rT{\cSVD{\fPsi}} \cSVD{\fPhi},
	  \end{align*}
	  which is technically equivalent to $\rT\fE \fE = \I{k}$ and $\rT\fE \Jtn \fE = \Z{k}$ (see \cref{eq:SymplOrthonROB}) for $\fE_{\textrm{CT}}$ and $\fE_{\textrm{cSVD}}$,
	  \item $\fe_1, \dots, \fe_k \in \R^{2n}$ are the basis vectors selected by the greedy algorithm,
	  \item $\widetilde{\fVE} \in \R^{2n \times 2k}$ is a ROB computed from CT or cSVD and $\fG \in \R^{2k \times r}$, $r \leq k$, stems from the coefficient matrix $\fC_{\fG} = [\fG \quad \TJtk \fG]$ computed by the NLP algorithm.
  \end{enumerate}
\end{Proposition}

\begin{proof}
  All of the listed methods derive a symplectic ROB of the form $\fVE = [\fE \quad \TJtn\fE]$ which satisfies \cref{eq:SymplOrthonROB}. By \cref{lem:SymplOrthonROB}, these ROBs are each a symplectic, orthonormal ROB. 
\end{proof}

In the following, we show that PSD Complex SVD is the solution of the PSD in the subset of symplectic, orthonormal ROBs. This was partly shown in \cite{Peng2016} which yet lacked the final step that, restricting to orthonormal, symplectic ROBs, a solution of $\PSD([\Xs \quad -\Jtn \Xs])$ solves $\PSD(\Xs)$ and vice versa. This proves that the PSD Complex SVD is not only near optimal in this set but indeed optimal. Furthermore, the proof we show is alternative to the original and naturally motivates an alternative formulation of the PSD Complex SVD which we call the POD of $\Ys$ in the following. To begin with, we reproduce the definition of PSD Complex SVD from \cite{Peng2016}.

\begin{Definition}[PSD Complex SVD]\label{def:ComplexSVD}
  We define the complex snapshot matrix
  \begin{align}\label{eq:ComplexSnapMat}
    &\Cs = [\fqs_1 + \imagUnit \fps_1, \dots, \fqs_{\ns} + \imagUnit \fps_{\ns}] \in \Cn^{n \times \ns},&
    &\xs_j = \begin{bmatrix} \fq_j\\ \fp_j \end{bmatrix} \text{for all } 1\leq j\leq \ns
  \end{align}
  which is derived with the imaginary unit $\imagUnit$.
  The PSD Complex SVD is a basis generation technique that requires the auxiliary complex matrix $\fUC \in \Cn^{n \times k}$ to fulfil
  \begin{align}\label{eq:ComplexPOD}
    &\minimize{\fUC \in \Cn^{n \times k}} \Fnorm{\Cs - \fUC \cTb\fUC \Cs}^2&
    &\textrm{subject to} \quad \cTb\fUC \fUC = \I{k}
  \end{align}
  and builds the actual ROB $\fVE \in \R^{2n \times 2k}$ with
  \begin{align*}
    &\fVE = [\fE \quad \TJtn \fE],&
    &\fE = \begin{bmatrix} \Reb{\fUC}\\ \Imb{\fUC} \end{bmatrix}.
  \end{align*}
  The solution of \cref{eq:ComplexPOD} is known to be based on the left-singular vectors of $\Cs$ which can be explicitly computed with a complex version of the SVD.
\end{Definition}

We emphasize that we denote this basis generation procedure as PSD Complex SVD in the following to avoid confusions with the usual complex SVD algorithm.

\begin{Proposition}[Minimizing PSD in the set of symplectic, orthonormal ROBs]\label{theo:PsdForOrthSymplROBs}
  Given the snapshot matrix $\Xs \in \R^{2n \times \ns}$ we augment this with ``rotated'' snapshots to $\Ys = [\Xs \quad \Jtn \Xs]$. We assume that $2k$ is such that we obtain a gap in the singular values of $\Ys$, i.e. $\sigma_{2k}(\Ys) > \sigma_{2k+1}(\Ys)$. Then, minimizing the PSD in the set of symplectic, orthonormal ROBs is equivalent to the following minimization problem
  \begin{align}\label{eq:PsdForOrthSymplROBs}
    &\minimize{\fV \in \R^{2n \times 2k}} \Fnorm{(\I{2n} - \fV \rT\fV) \begin{bmatrix} \Xs & \Jtn \Xs \end{bmatrix}}^2&
&\textrm{subject to} \quad\rT\fV \fV = \I{2k}.
  \end{align}
  Clearly, this is equivalent to the POD \cref{eq:POD} applied to the snapshot matrix $\Ys$. We, thus, call this procedure the POD of $\Ys$ in the following. A minimizer can be derived with the SVD as it is common for POD \cite{LuminyBook2017}.
\end{Proposition}

\begin{proof}
The proof proceeds in three steps: we show
\begin{enumerate}[label=(\roman*)]
  \item that $(\fu,\fv)$ is a pair of left- and right-singular vectors of $\Ys$ to the singular value $\sigma$ if and only if $(\TJtn \fu, \TJtns \fv)$ also is a pair of left- and right-singular vectors of $\Ys$ to the same singular value $\sigma$,
  \item that a solution of the POD of $\Ys$ is a symplectic, orthonormal ROB, i.e.\ $\fV = \fVE = [\fE \quad \TJtn \fE]$,
  \item that the POD of $\Ys$ is equivalent to the PSD for symplectic, orthonormal ROBs.
\end{enumerate}
  We start with the first step (i). Let $(\fu,\fv)$ be a pair of left- and right-singular vectors of $\Ys$ to the singular value $\sigma$. We use that the left-singular (or right-singular) vectors of $\Ys$ are a set of orthonormal eigenvectors of $\Ys\rT\Ys$ (or $\rT\Ys\Ys$). To begin with, we compute
  \begin{align} \label{eq:rotation}
  \begin{split}
    \TJtn \Ys \rT\Ys \Jtn
=&\; \TJtn (\Xs \rT\Xs + \Jtn\Xs\rT\Xs\TJtn) \Jtn
= \TJtn\Xs\rT\Xs\Jtn + \Xs\rT\Xs
= \Ys \rT\Ys,\\
    \TJtns \rT\Ys \Ys \Jtns
=&\; \TJtns \begin{bmatrix} \rT\Xs\Xs & \rT\Xs \Jtn\Xs\\ \rT\Xs\TJtn\Xs & \rT\Xs\Xs \end{bmatrix} \Jtns
= \begin{bmatrix} \rT\Xs\Xs & -\rT\Xs \TJtn\Xs\\ -\rT\Xs\Jtn\Xs & \rT\Xs\Xs \end{bmatrix}
= \rT\Ys \Ys
  \end{split}
  \end{align}
where we use $\TJtns = - \Jtns$. Thus, we can reformulate the eigenvalue problems of $\Ys\rT\Ys$ and, respectively, $\rT\Ys\Ys$ as
  \begin{align*}
    \sigma \fu
=&\; \Ys \rT\Ys \fu
= \Jtn \TJtn \Ys \rT\Ys \Jtn \TJtn \fu&
&\stackrel{\TJtn \cdot \vert}{\iff}&
\sigma \TJtn\fu
=&\; \TJtn \Ys \rT\Ys \Jtn \TJtn\fu
\stackrel{\cref{eq:rotation}}{=} \Ys \rT\Ys \TJtn\fu\\
    \sigma \fv
=&\; \rT\Ys \Ys \fv
= \Jtns \TJtns \rT\Ys \Ys \Jtns \TJtns \fv&
&\stackrel{\TJtns \cdot \vert}{\iff}&
\sigma \TJtns\fv
=&\; \TJtns \rT\Ys \Ys \Jtns \TJtns\fv
\stackrel{\cref{eq:rotation}}{=} \rT\Ys \Ys \TJtns\fv.
  \end{align*}
  Thus, $(\TJtn\fu, \TJtns\fv)$ is necessarily another pair of left- and right-singular vectors of $\Ys$ with the same singular value $\sigma$. We infer that the left-singular vectors $\fu_i$, $1\leq i\leq 2n$, ordered by the magnitude of the singular values in a descending order can be written as
\begin{align}\label{eq:PODOfYs}
   \fU = [\fu_1 \quad \TJtn \fu_1 \quad \fu_2 \quad \TJtn \fu_2 \quad \dots \fu_n \quad \TJtn \fu_n] \in \R^{2n \times 2n}.
\end{align}
  
  For the second step (ii), we remark that the solution of the POD is explicitly known to be any matrix which stacks in its columns $2k$ left-singular vectors of the snapshot matrix $\Ys$ with the highest singular value \cite{LuminyBook2017}. Due to the special structure \cref{eq:PODOfYs} of the singular vectors for the snapshot matrix $\Ys$, a minimizer of the POD of $\Ys$ necessarily adopts this structure. We are allowed to rearrange the order of the columns in this matrix and thus, the result of the POD of $\Ys$ can always be rearranged to the form
\begin{align*}
  \fVE = [\fE \quad \TJtn\fE],&
  &\fE = [\fu_1 \quad \fu_2 \quad \dots \quad \fu_k],&
  &\TJtn\fE = [\TJtn\fu_1 \quad \TJtn\fu_2 \quad \dots \quad \Jtn\fu_k].
\end{align*}
Note that it automatically holds that $\rT\fE \fE = \I{k}$ and $\rT\fE (\Jtn \fE) = \Z{k}$ since, in both products, we use the left-singular vectors from the columns of the matrix $\fU$ from \cref{eq:PODOfYs} which is known to be orthogonal from properties of the SVD. Thus, \cref{eq:SymplOrthonROB} holds and we infer from \cref{lem:SymplOrthonROB} that the POD of $\Ys$ indeed is solved by a symplectic, orthonormal ROB.
  
  For the final step (iii), we define the orthogonal projection operators
  \begin{align*}
    &\fPVE = \fVE \rTb\fVE = \fE \rT\fE + \TJtn\fE \rT\fE \Jtn,&
    &\fPVEoc = \I{2n} - \fPVE.
  \end{align*}
  Both are idempotent and symmetric, thus $\rTb\fPVEoc \fPVEoc = \fPVEoc \fPVEoc = \fPVEoc$. Due to $\Jtn\TJtn = \I{2n}$, it further holds 
  \begin{align*}
    \Jtn \rTb\fPVEoc \fPVEoc \TJtn = \Jtn \fPVEoc \TJtn = \Jtn\TJtn - \Jtn \fE \rT\fE \TJtn - \Jtn \TJtn\fE \rT\fE \Jtn \TJtn = \fPVEoc = \rTb\fPVEoc \fPVEoc.
  \end{align*}
  Thus, it follows
  \begin{align*}
    \Fnorm{\fPVEoc \Xs}^2
= \traceb{\rT\Xs \rTb\fPVEoc \fPVEoc \Xs}
= \traceb{\rT\Xs \Jtn \rTb\fPVEoc \fPVEoc \TJtn \Xs}
= \Fnorm{\fPVEoc \TJtn \Xs}^2
  \end{align*}
  and with $\Ys = [\Xs \quad \TJtn \Xs]$
  \begin{align*}
    2 \Fnorm{\fPVEoc \Xs}^2 = \Fnorm{\fPVEoc \Xs}^2 + \Fnorm{\fPVEoc \TJtn \Xs}^2
= \Fnorm{\fPVEoc [\Xs \quad \TJtn \Xs]}^2
= \Fnorm{\fPVEoc \Ys}^2,
  \end{align*}
  where we use in the last step that for two matrices $\fA \in \R^{2n \times u}$, $\fB \in \R^{2n \times v}$ for $u,v \in \N$, it holds $\Fnorm{\fA}^2 + \Fnorm{\fB}^2 = \Fnorm{[\fA \quad \fB]}^2$ for the Frobenius norm $\Fnorm{\bullet}$.
  
  Since it is equivalent to minimize a function $f: \R^{2n \times 2k} \rightarrow \R$ or a multiple $c f$ of it for any positive constant $c \in \Rpos$, minimizing $\Fnorm{\fPVEoc \Xs}^2$ is equivalent to minimizing $2\Fnorm{\fPVEoc \Xs}^2 = \Fnorm{\fPVEoc \Ys}^2$. Additionally, for a ROB of the form $\fVE = [\fE \quad \TJtn \fE]$ the constraint of orthonormal columns is equivalent to the requirements in \cref{eq:SymplOrthonROB}. Thus, to minimize the PSD in the class of symplectic, orthonormal ROBs is equivalent to the POD of $\Ys$ \cref{eq:PsdForOrthSymplROBs}.
\end{proof}

\begin{Remark} \label{rem:EqivPSD}
We remark that in the same fashion as the proof of step (iii) in \cref{theo:PsdForOrthSymplROBs}, it can be shown that, restricting to symplectic, orthonormal ROBs, a solution of $\PSD([\Xs \quad \Jtn\Xs])$ is a solution of $\PSD(\Xs)$ and vice versa, which is one detail that was missing in \cite{Peng2016} to show the optimality of PSD Complex SVD in the set of symplectic, orthonormal ROBs.
\end{Remark}

We next prove that PSD Complex SVD is equivalent to POD of $\Ys$ from \cref{eq:PsdForOrthSymplROBs} and thus, also minimizes the PSD in the set of symplectic, orthonormal bases. To this end, we repeat the optimality result from \cite{Peng2016} and extend it with the results of the present paper.

\begin{Proposition}[Optimality of PSD Complex SVD]
\label{theo:OptimalityCSVD}
Let $\mathbb{M}_2 \subset \R^{2n \times 2k}$ denote the set of symplectic bases with the structure $\fVE = [\fE \quad \TJtn\fE]$. The PSD Complex SVD solves $\PSD([\Xs \quad -\Jtn\Xs])$ in $\mathbb{M}_2$.
\end{Proposition}

\begin{proof}
  See \cite[Theorem 4.5.]{Peng2016}.
\end{proof}

\begin{Proposition}[Equivalence of POD of $\Ys$ and PSD Complex SVD]\label{prop:EquivalenceComplexSVD}
  PSD Complex SVD is equivalent to the POD of $\Ys$. Thus, PSD Complex SVD yields a minimizer of the PSD for symplectic, orthonormal ROBs.
\end{Proposition}
\begin{proof}
  By \cref{theo:OptimalityCSVD}, PSD Complex SVD minimizes \cref{eq:PsdForOrthSymplROBs} in the set $\mathbb{M}_2$ of symplectic bases with the structure $\fVE = [\fE \quad \TJtn\fE]$. Thus, \cref{eq:SymplROBForEF} holds with $\fF = \TJtn \fE$ which is equivalent to the conditions on $\fE$ required in \cref{eq:SymplOrthonROB}. By \cref{lem:SymplOrthonROB}, we infer that $\mathbb{M}_2$ equals the set of symplectic, orthonormal bases.
  
  Furthermore, we can show that, in the set $\mathbb{M}_2$, a solution of $\PSD([\Xs \quad -\Jtn\Xs])$ is a solution of $\PSD(\Xs)$ and vice versa (see \cref{rem:EqivPSD}). Thus, PSD Complex SVD minimizes the PSD for the snapshot matrix $\Xs$ in the set of orthonormal, symplectic matrices and PSD Complex SVD and the POD of $\Ys$ solve the same minimization problem.
\end{proof}

We emphasize that the computation of a minimizer of \cref{eq:PsdForOrthSymplROBs} via PSD Complex SVD requires less memory storage than the computation via POD of $\Ys$. The reason is that the complex formulation uses the complex snapshot matrix $\Cs \in \Cn^{n \times \ns}$ which equals $2 \cdot n \cdot \ns$ floating point numbers while the solution with the POD of $\Ys$ method artificially enlarges the snapshot matrix to $\Ys \in \R^{2n \times 2\ns}$ which are $4 \cdot n \cdot \ns$ floating point numbers. Still, the POD of $\Ys$ might be computationally more efficient since it is a purely real formulation and thereby does not require complex arithmetic operations.

\subsection{Symplectic, non-orthonormal basis generation}
\label{sec:NonOrthPSD}
In the next step, we want to give an idea how to leave the class of symplectic, orthonormal ROBs. We call a basis generation technique symplectic, non-orthonormal if it is able to compute a symplectic, non-orthonormal basis.

In \cref{lem:OrtonSymplBasisGen}, we briefly showed that most existing symplectic basis generation techniques generate a symplectic, orthonormal ROB. The only exception is the NLP algorithm suggested in \cite{Peng2016}. It is able to compute a non-orthonormal, symplectic ROB. The algorithm is based on a given initial guess $\fV_0 \in \R^{2n \times 2k}$ which is a symplectic ROB e.g.\ computed with PSD Cotangent Lift or PSD Complex SVD. Nonlinear programming is used to leave the class of symplectic, orthonormal ROBs and derive an optimized symplectic ROB $\fV = \fV_0 \fC$ with the symplectic coefficient matrix $\fC \in \R^{2k \times 2r}$ for some $r \leq k$. Since this procedure searches a solution spanned by the columns of $\fV_0$, it is not suited to compute a global optimum of the PSD which we are interested in the scope of this paper.

In the following, we present a new basis generation technique that is based on an SVD-like decomposition for matrices $\fB \in \R^{2n \times m}$ presented in \cite{Xu2003}. To this end, we introduce this decomposition in the following.

\begin{Proposition}[SVD-like decomposition \cite{Xu2003}]\label{lem:SVDlikeDecomp}
Any real matrix $\fB \in \R^{2n \times m}$ can be decomposed as the product of a symplectic matrix $\fS \in \R^{2n \times 2n}$, a sparse and potentially non-diagonal matrix $\fD \in \R^{2n \times m}$ and an orthogonal matrix $\fQ \in \R^{m \times m}$ with
\begin{align} \label{eq:SVDlikeDecomp}
  &\fB = \fS \fD \fQ,&
  &\fD =  
  \begin{blockarray}{
    >{\hspace{\arraycolsep}\idxsize}c*{3}{>{\idxsize}c}<{\hspace{\arraycolsep}}
    >{\idxsize}c
  }
    p & q & p & m-2p-q \\
    \begin{block}{
      (>{\hspace{\arraycolsep}}cccc<{\hspace{\arraycolsep}})
      >{\idxsize}c
    }
      \fSigmaS & \Z{} & \Z{} & \Z{} & p \\
      \Z{}     & \I{} & \Z{} & \Z{} & q \\
      \Z{}     & \Z{} & \Z{} & \Z{} & n-p-q \\
      \Z{}     & \Z{} & \fSigmaS & \Z{} & p \\
      \Z{}     & \Z{} & \Z{} & \Z{} & q \\
      \Z{}     & \Z{} & \Z{} & \Z{} & n-p-q \\
    \end{block}
  \end{blockarray},
  &\begin{split}
    \fSigmaS = \diag(\sigmaS_1, \dots, \sigmaS_p) \in \R^{p \times p},\\
    \sigmaS_i > 0 \quad\text{for } 1\leq i\leq p.
  \end{split}
\end{align}
with $p,q \in \N$ and $\rank(\fB) = 2p+q$ and where we indicate the block row and column dimensions in $\fD$ by small letters. The diagonal entries $\sigmaS_i$, $1 \leq i\leq p$, of the matrix $\fSigmaS$ are related to the pairs of purely imaginary eigenvalues $\lambda_j(\fM), \lambda_{p+j}(\fM) \in \Cn$ of $\fM = \rT\fB \Jtn \fB \in \R^{m \times m}$ with 
\begin{align*}
  &\lambda_j(\fM) = -(\sigmaS_j)^2\imagUnit,&
  &\lambda_{p+j}(\fM) = (\sigmaS_j)^2\imagUnit,&
  &1 \leq j\leq p.
\end{align*}
\end{Proposition}

\begin{Remark}[Singular values]\label{rem:SingVal}
We call the diagonal entries $\sigmaS_i$, $1 \leq i\leq p$, of the matrix $\fSigmaS$ from \cref{lem:SVDlikeDecomp} in the following the symplectic singular values. The reason is the following analogy to the classical SVD.

The classical SVD decomposes $\fB \in \R^{2n \times m}$ as $\fB = \fU \fSigma \rT\fV$ where $\fU \in \R^{2n \times 2n}$, $\fV \in \R^{m \times m}$ are each orthogonal matrices and $\fSigma \in \R^{2n \times m}$ is a diagonal matrix with the singular values $\sigma_i$ on its diagonal $\diag(\fSigma) = [\sigma_1, \dots, \sigma_r, 0, \dots, 0] \in \R^{\min(2n,m)}$, $r = \rank(\fB)$. The singular values are linked to the real eigenvalues of $\fN = \rT\fB \fB$ with $\lambda_i(\fN) = \sigma_i^2$. Furthermore, for the SVD, it holds due to the orthogonality of $\fU$ and $\fV$, respectively, $\rT\fB \fB = \rT\fV \fSigma^2 \fV$ and $\fB \rT\fB = \rT\fU \fSigma^2 \fU$.

A similar relation can be derived for an SVD-like decomposition from \cref{lem:SVDlikeDecomp}. Due to the structure of the decomposition \cref{eq:SVDlikeDecomp} and the symplecticity of $\fS$, it holds
\begin{align} \label{eq:SymplSingVal}
  \begin{split}
	  \rT\fB \Jtn \fB
	  =&\; \rT\fQ \rT\fD \overbrace{\rT\fS \Jtn \fS}^{=\Jtn} \fD \fQ\\
	  =&\; \rT\fQ \rT\fD \Jtn \fD \fQ,
	\end{split}
  &\rT\fD \Jtn \fD =
  \begin{blockarray}{
    >{\hspace{\arraycolsep}\idxsize}c*{3}{>{\idxsize}c}<{\hspace{\arraycolsep}}
    >{\idxsize}c
  }
    p & q & p & m-2p-q & \\
    \begin{block}{
      (>{\hspace{\arraycolsep}}cccc<{\hspace{\arraycolsep}})
      >{\idxsize}c
    }
			  \Z{}        & \Z{} & \fSigmaS^2 & \Z{} & p\\
			  \Z{}        & \Z{} & \Z{}       & \Z{} & q\\
			  -\fSigmaS^2 & \Z{} & \Z{}       & \Z{} & p\\
			  \Z{}        & \Z{} & \Z{}       & \Z{} & m-2p-q\\
    \end{block}
  \end{blockarray}.
\end{align}
This analogy is why we call the diagonal entries $\sigmaS_i$, $1 \leq i\leq p$, of the matrix $\fSigmaS$ symplectic singular values.
\end{Remark}

The idea for the basis generation now is to select $k \in \N$ pairs of columns of $\fS$ in order to compute a symplectic ROB. The selection should be based on the importance of these pairs which we characterize by the following proposition by linking the Frobenius norm of a matrix with the symplectic singular values.

\begin{Proposition}\label{prop:FrobSymplSingVal}
Let $\fB \in \R^{2n \times m}$ with an SVD-like decomposition $\fB =  \fS \fD \fQ$ with $p,q \in \N$ from \cref{lem:SVDlikeDecomp}. The Frobenius norm of $\fB$ can be rewritten as
\begin{align}\label{eq:FrobSymplSingVal}
  &\Fnorm{\fB}^2 = \sum_{i=1}^{p+q} (\wsigmaS_i)^2,&
  &\wsigmaS_i = \begin{cases}
    \sigmaS_i \sqrt{\tnorm{\fs_{i}}^2 + \tnorm{\fs_{n+i}}^2}, & 1\leq i\leq p,\\
    \tnorm{\fs_{i}}, & p+1 \leq i\leq p+q
  \end{cases}
\end{align}
where $\fs_i \in \R^{2n}$ is the $i$-th column of $\fS$ for $1\leq i\leq 2n$. In the following, we refer to each $\wsigmaS_i$ as the weighted symplectic singular value.
\end{Proposition}

\begin{proof}
  We insert the SVD-like decomposition $\fB = \fS \fD \fQ$ and use the orthogonality of $\fQ$ to reformulate
  \begin{align*}
    \Fnorm{\fB}^2 
=&\; \Fnorm{\fS \fD \fQ}^2
= \Fnorm{\fS \fD}^2
= \trace(\rT\fD \rT\fS \fS \fD)
= \sum_{i=1}^{p} (\sigmaS_i)^2 \rT\fs_i \fs_i + \sum_{i=1}^{p} (\sigmaS_i)^2 \rT\fs_{n+i} \fs_{n+i} + \sum_{i=1}^q \rT\fs_{p+i} \fs_{p+i}\\
  =&\; \sum_{i=1}^p (\sigmaS_i)^2 \lb \tnorm{\fs_{i}}^2 + \tnorm{\fs_{n+i}}^2 \rb
+ \sum_{i=1}^q \tnorm{\fs_{p+i}}^2
  \end{align*}
  which is equivalent to \cref{eq:FrobSymplSingVal}.
\end{proof}

It proves true in the following \cref{prop:PSDDecay} that we can delete single addends $\wsigmaS_i$ in \cref{eq:FrobSymplSingVal} with the symplectic projection used in the PSD if we include the corresponding pair of columns in the ROB. This will be our selection criterion in the new basis generation technique that we denote PSD SVD-like decomposition.

\begin{Definition}[PSD SVD-like decomposition]\label{def:PSDSVDlike}
We compute an SVD-like decomposition \cref{eq:SVDlikeDecomp} as $\Xs = \fS \fD \fQ$ of the snapshot matrix $\Xs \in \R^{2n \times \ns}$ and define $p,q \in \N$ as in \cref{lem:SVDlikeDecomp}. In order to compute a ROB $\fV$ with $2k$ columns, find the $k$ indices $i \in \idxSetPSD = \{i_1, \dots, i_k\} \subset \{1, \dots, p+q\}$ which have large contributions $\wsigmaS_i$ in \cref{eq:FrobSymplSingVal} with
\begin{align} \label{eq:IdxPSDSVDlike}
  \idxSetPSD = \argmax_{\substack{\mathcal{I} \subset \{1, \dots, p+q\}\\ \abs{\mathcal{I}} = k}}
\lb \sum_{i \in \mathcal{I}} \lb \wsigmaS_i\rb^2 \rb.
\end{align}
To construct the ROB, we choose the $k$ pairs of columns $\fs_i \in \R^{2n}$ from $\fS$ corresponding to the selected indices $\idxSetPSD$ such that
\begin{align*}
  \fV = [\fs_{i_1}, \dots, \fs_{i_k}, \fs_{n+i_1}, \dots, \fs_{n+i_k}] \in \R^{2n \times 2k}.
\end{align*}
\end{Definition}

The special choice of the ROB is motivated by the following theoretical result which is very analogous to the results known for the classical POD in the framework of orthogonal projections.

\begin{Proposition}[Projection error by neglegted weighted symplectic singular values]\label{prop:PSDDecay}
Let $\fV \in \R^{2n \times 2k}$ be a ROB constructed with the procedure described in \cref{def:PSDSVDlike} to the index set $\idxSetPSD \subset \{1, \dots, p+q\}$ with $p,q \in \N$ from \cref{lem:SVDlikeDecomp}. The PSD functional can be calculated by
\begin{align} \label{eq:PSDDecay}
  \Fnorm{(\I{2n} - \fV\si\fV) \Xs}^2
= \sum_{i \in \{1, \dots, p+q\} \setminus \idxSetPSD} \lb \wsigmaS_i \rb^2,
\end{align}
which is the cumulative sum of the squares of the neglected weighted symplectic singular values.
\end{Proposition}

\begin{proof}
  Let $\fV \in \R^{2n \times 2k}$ be a ROB constructed from an SVD-like decomposition $\Xs = \fS \fD \fQ$ of the snapshot matrix $\Xs \in \R^{2n \times 2k}$ with the procedure described in \cref{def:PSDSVDlike}. Let $p,q \in \N$ be defined as in \cref{lem:SVDlikeDecomp} and $\idxSetPSD = \{i_1, \dots, i_k\} \subset \{1, \dots, p+q\}$ be the set of indices selected with \cref{eq:IdxPSDSVDlike}.
  
  For the proof, we introduce a slightly different notation of the ROB $\fV$. The selection of the columns $\fs_i$ of $\fS$ is denoted with the selection matrix $\I{\idxSetPSDtk} \in \R^{2n \times 2k}$ based on
  \begin{align*}
    &\lb \I{\idxSetPSD} \rb_{\alpha,\beta} =
    \begin{cases}
      1,& \alpha=i_\beta \in \idxSetPSD\\
      0,& \text{else}
    \end{cases}
    &\text{for } \quad
    \begin{split}
      &1\leq \alpha \leq 2n,\\
      &1\leq \beta \leq k,
    \end{split}
    &\I{\idxSetPSDtk} = [\I{\idxSetPSD}, \; \TJtn\I{\idxSetPSD}].
  \end{align*}
  which allows us to write the ROB as the matrix--matrix product $\fV = \fS \I{\idxSetPSDtk}$. Furthermore, we can select the neglected entries with $\I{2n} - \I{\idxSetPSDtk} \rTb{\I{\idxSetPSDtk}}$.

We insert the SVD-like decomposition and the representation of the ROB introduced in the previous paragraph in the PSD which reads
  \begin{align*}
     \Fnorm{(\I{2n} - \fV\si\fV) \Xs}^2
= \Fnorm{(\I{2n} - \fS \I{\idxSetPSDtk} \TJtk \rT{\I{\idxSetPSDtk}} \rT\fS \Jtn) \fS \fD \fQ}^2
= \Big\lVert\fS (\I{2n} - \I{\idxSetPSDtk} \TJtk \rT{\I{\idxSetPSDtk}} \overbrace{\rT\fS \Jtn \fS}^{=\Jtn} ) \fD \Big\rVert^2_{\mathrm{F}}
  \end{align*}
where we use the orthogonality of $\fQ$ and the symplecticity of $\fS$ in the last step. We can reformulate the product of Poisson matrices and the selection matrix as
\begin{align*}
  \TJtk \rT{\I{\idxSetPSDtk}} \Jtn
= 
\TJtk
\begin{bmatrix}
  \rT{\I{\idxSetPSD}}\\
  \rT{\I{\idxSetPSD}}\Jtn
\end{bmatrix}
\Jtn
=
\begin{bmatrix}
  \Z{k}  & -\I{k}\\
  \I{k} & \Z{k}
\end{bmatrix}
\begin{bmatrix}
  \rT{\I{\idxSetPSD}}\Jtn\\
  -\rT{\I{\idxSetPSD}}
\end{bmatrix}
= \rT{\I{\idxSetPSDtk}}.
\end{align*}
Thus, we can further reformulate the PSD as
\begin{align*}
  \Fnorm{(\I{2n} - \fV\si\fV) \Xs}^2
= \Fnorm{\fS \lb \I{2n} - \I{\idxSetPSDtk} \rTb{\I{\idxSetPSDtk}} \rb \fD}^2
= \sum_{i \in \{1, \dots, p+q\} \setminus \idxSetPSD} \lb \wsigmaS_i \rb^2
\end{align*}
where $\wsigmaS_i$ are the weighted symplectic singular values from \cref{eq:FrobSymplSingVal}. In the last step, we use that the resultant diagonal matrix in the braces sets all rows of $\fD$ with indices $i, n+i$ to zero for $i \in \idxSetPSD$. Thus, the last step can be concluded analogously to the proof of \cref{prop:FrobSymplSingVal}.
\end{proof}

A direct consequence of \cref{prop:PSDDecay} is that the decay of the PSD functional is proportional to the decay of the sum over the neglected weighted symplectic singular values $\wsigmaS_i$ from \cref{eq:FrobSymplSingVal}. In the numerical example \cref{subsubsec:NumExp:projErr}, we observe an exponential decrease of this quantities which induces an exponential decay of the PSD functional.

\begin{Remark}[Computation of the SVD-like decomposition]\label{rem:CompSVDlikeDecomp}
To compute an SVD-like decompostion \cref{eq:SVDlikeDecomp} of $\fB$, several approaches exist. The original paper \cite{Xu2003} derives a decomposition based on the product $\rT\fB \Jtn \fB$ which is not good for a numerical computation since errors can arise from cancellation. In \cite{Xu2005}, an implicit version is presented that does not require the computation of the full product $\rT\fB \Jtn \fB$ but derives the decomposition implicitly by transforming $\fB$. Furthermore, \cite{Agoujil2018} introduces an iterative approach to compute an SVD-like decomposition which computes parts of an SVD-like decomposition with a block-power iterative method. In the present case, we use the implicit approach \cite{Xu2005}.
\end{Remark}

\subsection{Interplay of non-orthonormal and orthonormal ROBs}
We give further results on the interplay of non-orthonormal and orthonormal ROBs. The fundamental statement in the current section is the Orthogonal SR decomposition \cite{Bunse1986,Xu2003}.

\begin{Proposition}[Orthogonal SR decomposition] \label{prop:OrthoSRDecomp}
For each matrix $\fB \in \R^{2n \times m}$ with $m \leq n$, there exists a symplectic, orthogonal matrix $\fS \in \R^{2n \times 2n}$, an upper triangular matrix $\fR_{11} \in \R^{m \times m}$ and a strictly upper triangular matrix $\fR_{21} \in \R^{m \times m}$ such that
\begin{align*}
  &\fB = \fS \begin{bmatrix} \fR_{11}\\ \Z{(n-m) \times m}\\ \fR_{21}\\ \Z{(n-m) \times m} \end{bmatrix}
= [\fS_m \quad \TJtn \fS_m] \begin{bmatrix}  \fR_{11}\\ \fR_{21} \end{bmatrix},&
  &\begin{split}
    \fS_i =&\; [\fs_1, \dots, \fs_i], \quad 1\leq i\leq n,\\
    \fS =&\; [\fs_1, \dots, \fs_n, \TJtn \fs_1, \dots, \TJtn \fs_n].
  \end{split}
\end{align*}
\end{Proposition}
We remark that a similar result can be derived for the case $m > n$ \cite{Xu2003} but it is not introduced since we do not need it in the following.

\begin{proof}
Let $\fB \in \R^{2n \times m}$ with $m \leq n$. We consider the QR decomposition
\begin{align*}
  \fB = \fQ \begin{bmatrix} \fR\\ \Z{(2n-m) \times m} \end{bmatrix}
\end{align*}
where $\fQ \in \R^{2n \times 2n}$ is an orthogonal matrix and $\fR \in \R^{2n \times m}$ is upper triangular. The original Orthogonal SR decomposition \cite[Corollary 4.5.]{Bunse1986} for the square matrix states that we can decompose $\fQ \in \R^{2n \times 2n}$ as a symplectic, orthogonal matrix $\fS \in \R^{2n \times 2n}$, an upper triangular matrix $\fRs_{11} \in \R^{n \times n}$, a strictly upper triangular matrix $\fRs_{21} \in \R^{n \times n}$ and two (possibly) full matrices $\fRs_{12}, \fRs_{22} \in \R^{n \times n}$
\begin{align*}
  &\fQ = \fS \begin{bmatrix} \fRs_{11} & \fRs_{12}\\ \fRs_{21} & \fRs_{22} \end{bmatrix}&
&\text{and thus}&
&\fB
= \fS \begin{bmatrix} \fRs_{11} & \fRs_{12}\\ \fRs_{21} & \fRs_{22} \end{bmatrix} \begin{bmatrix} \fR\\ \Z{(2n-m) \times m} \end{bmatrix}
= \fS \begin{bmatrix} \fRs_{11}\\ \fRs_{21}\end{bmatrix} \begin{bmatrix} \fR\\ \Z{(n-m) \times m} \end{bmatrix}.
\end{align*}
Since $\fR$ is upper triangular, it does preserve the (strictly) upper triangular pattern in $\fRs_{11}$ and $\fRs_{21}$ and we obtain the (strictly) upper triangular matrices $\fR_{11}, \fR_{21} \in \R^{m \times m}$ from
\begin{align*}
 \begin{bmatrix} \fR_{11}\\ \Z{(n-m) \times m}\\ \fR_{21}\\ \Z{(n-m) \times m} \end{bmatrix}
= \begin{bmatrix} \fRs_{11}\\ \fRs_{21}\end{bmatrix} \begin{bmatrix} \fR\\ \Z{(n-m) \times m} \end{bmatrix}.
\end{align*}
\end{proof}

Based on the Orthogonal SR decomposition, the following two propositions prove bounds for the projection errors of PSD which allows an estimate for the quality of the respective method. In both cases we require the basis size to satisfy $k \leq n$ or $2k \leq n$, respectively. This restriction is not limiting in the context of symplectic MOR as in all application cases $k \ll n$.

\begin{Proposition}\label{prop:EstimatePODoPSD}
Let $\fV \in \R^{2n \times k}$ be a minimizer of POD with $k \leq n$ basis vectors and $\fVE \in \R^{2n \times 2k}$ be a minimizer of the PSD in the class of orthonormal, symplectic matrices with $2k$ basis vectors. Then, the orthogonal projection errors of $\fVE$ and $\fV$ satisfy
\begin{align*}
  \Fnorm{(\I{2n} - \fVE \rT\fVE) \Xs}^2 \leq \Fnorm{\lb \I{2n} - \fV \rT\fV \rb \Xs}^2.
\end{align*}
\end{Proposition}

\begin{proof}
The Orthogonal SR decomposition (see \cref{prop:OrthoSRDecomp}) guarantees that a symplectic, orthogonal matrix $\fS \in \R^{2n \times 2k}$ and $\fR \in \R^{2k \times k}$ exist with $\fV = \fS \fR$. Since both matrices $\fV$ and $\fS$ are orthogonal and $\img(\fV) \subset \img(\fS)$, we can show that $\fS$ yields a lower projection error than $\fV$ with
\begin{align*}
  \Fnorm{\lb \I{2n} - \fS \rT\fS \rb \Xs}^2
  =&\; \Fnorm{\lb\I{2n} - \fS \rT\fS\rb \lb\I{2n} - \fV \rT\fV\rb \Xs}^2
  = \sum_{i=1}^{\ns} \tnorm{\lb\I{2n} - \fS \rT\fS\rb \lb\I{2n} - \fV \rT\fV\rb \xs_i}^2\\
  \leq&\; \underbrace{\tnorm{\I{2n} - \fS \rT\fS}^2}_{\leq 1} \; \sum_{i=1}^{\ns} \tnorm{\lb\I{2n} - \fV \rT\fV\rb \xs_i}^2
  \leq  \Fnorm{\lb\I{2n} - \fV \rT\fV\rb \Xs}^2
\end{align*}

Let $\fVE \in \R^{2n \times 2k}$ be a minimizer of the PSD in the class of symplectic, orthonormal ROBs. By definition of $\fVE$, it yields a lower projection error than $\fS$. Since both ROBs are symplectic, orthonormal, we can exchange the symplectic inverse with the transposition (see \cref{lem:SymplOrthonROB}, (iii)). This proves the assertion with
\begin{align*}
  \Fnorm{\lb \I{2n} - \fV \rT\fV \rb \Xs}^2
\geq \Fnorm{\lb\I{2n} - \fS \rT\fS\rb \Xs}^2
\geq \Fnorm{\lb \I{2n} - \fVE \rT\fVE\rb \Xs}^2.
\end{align*}
\end{proof}

\cref{prop:EstimatePODoPSD} proves that we require at most twice the number of basis vectors to generate a symplectic, orthonormal basis with an orthogonal projection error at least as small as the one of the classical POD. An analogous result can be derived in the framework of a symplectic projection which is proven in the following proposition.

\begin{Proposition} \label{prop:EstimateoPSDPSD}
Assume there exists a minimizer $\fV \in \R^{2n \times 2k}$ of the general PSD for a basis size $2k \leq n$ with potentially non-orthonormal columns. Let $\fVE \in \R^{2n \times 4k}$ be a minimizer of the PSD in the class of symplectic, orthogonal bases of size $4k$. Then, we know that the symplectic projection error of $\fVE$ is less than or equal to the one of $\fV$, i.e.\
\begin{align*}
  \Fnorm{(\I{2n} - \fVE \si\fVE) \Xs}^2 \leq \Fnorm{(\I{2n} - \fV \si\fV) \Xs}^2.
\end{align*}
\end{Proposition}

\begin{proof}
  Let $\fV \in \R^{2n \times 2k}$ be a minimizer of PSD with $2k \leq n$. By \cref{prop:OrthoSRDecomp}, we can determine a symplectic, orthogonal matrix $\fS \in \R^{2n \times 4k}$ and $\fR \in \R^{4k \times 2k}$ with $\fV = \fS \fR$. Similar to the proof of \cref{prop:EstimatePODoPSD}, we can bound the projection errors. We require the identity
  \begin{align*}
    (\I{2n} - \fS \si\fS) (\I{2n} - \fV \si\fV)
= \I{2n} - \fS \si\fS - \fV \si\fV + \underbrace{\fS \overbrace{\si\fS \fS}^{=\I{4k}} \fR}_{=\fV} \underbrace{\TJtk \rT\fR \rT\fS \Jtn}_{=\si\fV}
= \I{2n} - \fS \si\fS.
  \end{align*}
  With this identity, we proceed analogously to the proof of \cref{prop:EstimatePODoPSD} and derive for a minimizer $\fVE \in \R^{2n \times 4k}$ of PSD in the class of symplectic, orthonormal ROBs
  \begin{align*}
    \Fnorm{(\I{2n} - \fVE \si\fVE) \Xs}^2
    \leq&\; \Fnorm{(\I{2n} - \fS \si\fS) \Xs}^2
= \Fnorm{(\I{2n} - \fS \si\fS) (\I{2n} - \fV \si\fV) \Xs}^2\\
\leq&\; \underbrace{\tnorm{(\I{2n} - \fS \si\fS)}^2}_{\leq 1} \Fnorm{(\I{2n} - \fV \si\fV) \Xs}^2
\leq \Fnorm{(\I{2n} - \fV \si\fV) \Xs}^2.
  \end{align*}
\end{proof}

\cref{prop:EstimateoPSDPSD} proves that we require at most twice the number of basis vectors to generate a symplectic, orthonormal basis with a symplectic projection error at least as small as the one of a (potentially non-orthonormal) minimizer of PSD.

\section{Numerical results}
\label{sec:NumRes}

The numerical experiments in the present paper are based on a two-dimensional plane strain linear elasticity model which is described by a Lam\'e--Navier equation
\begin{align*}
	\rhoz \sodeldel{t^2}{\fu(\fxi, t, \fmu)} - \lamemu \, \Delta_{\fxi} \fu(\fxi, t, \fmu) + (\lamelambda + \lamemu) \grad[\fxi] \lb \divb[\fxi]{\fu(\fxi, t, \fmu)} \rb = \rhoz \, \fg(\fxi, t).
\end{align*}
for $\fxi \in \domain \subset \R^2$ and $t \in \ftInterval$ with the density $\rhoz \in \Rpos$, the Lam\'e constants $ \fmu = (\lamelambda, \lamemu) \in \Rpos^2$, the external force $\fg: \domain \times \ftInterval \rightarrow \R^2$ and Dirichlet boundary conditions on $\boundaryDirichlet \subset \boundary := \partial\varOmega$ and Neumann boundary conditions on $\boundaryNeumann \subset \boundary$. We apply non-dimensionalization (e.g.\ \cite[Chapter 4.1]{Langtangen2016}), apply the Finite Element Method (FEM) with first-order Lagrangian elements on a triangular mesh and rewrite the system as first-order system to arrive at the parametric linear system \cref{eq:LinSys} with
\begin{align}\label{eq:HamLinElast}
  &\fx(t, \fmu) = \begin{bmatrix} \fq(t, \fmu)\\ \fp(t, \fmu) \end{bmatrix},&
  &\fH(\fmu) = \begin{bmatrix} \fK(\fmu) & \Z{n}\\ \Z{n} & \inv\fM \end{bmatrix},&
&\fh(t) = \begin{bmatrix} -\ff(t)\\ \Z{n \times 1} \end{bmatrix}
\end{align}
where $\fq(t, \fmu) \in \R^n$ is the vector of displacement DOFs, $\fp(t, \fmu) \in \R^n$ is the vector of linear momentum DOFs, $\fK(\fmu) \in \R^{n \times n}$ is the stiffness matrix, $\inv\fM \in \R^{n \times n}$ is the inverse of the mass matrix and $\ff(t, \fmu)$ is the vector of external forces.

We remark that a Hamiltonian formulation with the velocity DOFs $\fv(t) = \dd{t} \fx(t) \in \R^n$ instead of the linear momentum DOFs $\fp(t)$ is possible if a non-canonical symplectic structure is used. Nevertheless, in \cite[Remark 3.8.]{Peng2016} it is suggested to switch to a formulation with a canonical symplectic structure for the MOR of Hamiltonian systems.

In order to solve the system \cref{eq:HamLinElast} numerically with a time-discrete approximation $\fx_i(\fmu) \approx \fx(t_i, \fmu)$ for each of $\nt \in \N$ time steps $t_i \in \ftInterval$, $1\leq i\leq \nt$, a numerical integrator is required. The preservation of the symplectic structure in the time-discrete system requires a so-called symplectic integrator \cite{Hairer2006,Bhatt2017}. In the context of our work, the implicit midpoint scheme is used in all cases.

\begin{Remark}[Modified Hamiltonian]\label{rem:ModHam}
We remark that, even though the symplectic structure is preserved by symplectic integrators, the Hamiltonian may be modified in the time-discrete system compared to the original Hamiltonian. In the case of a quadratic Hamiltonian (see \cref{cor:QuadHam}) and a symplectic Runge-Kutta integrator, the modified Hamiltonian equals the original Hamiltonian since these integrators preserve quadratic first integrals. For further details, we refer to \cite[Chapter IX.]{Hairer2006} or \cite[Section 5.1.2 and 5.2]{Leimkuhler2005}.
\end{Remark}

The model parameters are the first and second Lam\'e constants with $\fmu = (\lamelambda, \lamemu) \in \paramSet = [35 \cdot 10^{9}, 125 \cdot 10^{9}] \;\txtfrac{\usiN}{\usim^2} \times [35 \cdot 10^{9}, 83 \cdot 10^{9}] \;\txtfrac{\usiN}{\usim^2}$ which varies between cast iron and steel with approx.\ $12\%$ chromium \cite[App.\ E 1 Table 1]{Dubbel2014ChapE}. The density is set to $\rhoz = 7856 \;\txtfrac{\usikg}{\usim^3}$. The non-dimensionalization constants are set to $\nondim\lamelambda = \nondim\lamemu = 81 \cdot 10^{9} \;\txtfrac{\usiN}{\usim^2}$, $\nondim\xi = 1 \;\usim$, $\nondim{g} = 9.81\;\txtfrac{\usim}{\usis^2}$. The geometry is a simple cantilever beam clamped on the left side with a force applied to the right boundary. The time interval is chosen to be $t \in [\tInit, \tEnd]$ with $\tInit = 0 \,\usis$ and $\tEnd = 7.2 \cdot 10^{-2} \,\usis$ which is one oscillation of the beam. For the numerical integration $\nt = 151$ time steps are used.

\begin{figure}[h!]
\centering
\includegraphics[scale=1.0]{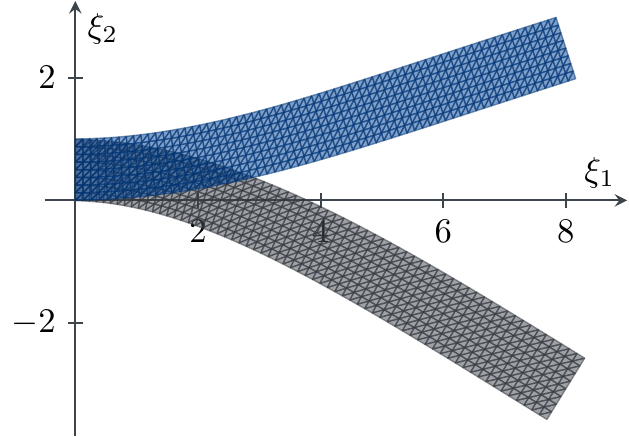}
\caption{an exaggerated illustration of the displacements $\fq(t, \fmu)$ of the non-autonomous beam model (a)~at the time  with the maximum displacement (gray) and (b)~at the final time (blue).}
\end{figure}

The symplectic MOR techniques examined are PSD Complex SVD (\cref{def:ComplexSVD}), the greedy procedure \cite{Maboudi2017} and the newly introduced PSD SVD-like decomposition (\cref{def:PSDSVDlike}). The MOR techniques that do not necessarily derive a symplectic ROB are called non-symplectic MOR techniques in the following. The non-symplectic MOR techniques investigated in the scope of our numerical results are the POD applied to the full state $\fx(t, \fmu)$ (POD full state) and a POD applied to the displacement $\fq(t, \fmu)$ and linear momentum states $\fp(t, \fmu)$ separately (POD separate states). To summarize the basis generation methods, let us enlist them in \Cref{tab:basis_gen} where $\prcSVD(\bullet)$ and $\prcCSVD(\bullet)$ denote the SVD and the complex SVD, respectively.

\begin{table}[htbp]
\centering
{\renewcommand{\arraystretch}{1.5}
\begin{tabular}{p{2.2cm}p{5.6cm}lcc}
\hline\noalign{\smallskip}
method & solution & solution procedure & ortho- & sympl. \\[-.6em]
&&&norm.\\
\hline\noalign{\smallskip}
POD full &
  $\fV_k = \fU(:,1:k)$ & $\fU = \prcSVD(\Xs)$ &
  \cmark &
  \xmark \\[1em]
POD separate &
  \multirow{2}{*}{$\fV_k = \begin{bmatrix} \fU_{\fp}(:,1:k)\\ \fU_{\fq}(:,1:k) \end{bmatrix}$} &
  $\fU_{\fp} = \prcSVD \lb [\fp_1, \dots, \fp_{\ns}] \rb$ &
  \cmark & 
  \xmark \\
&& $\fU_{\fq} = \prcSVD \lb [\fq_1, \dots, \fq_{\ns}] \rb$\\[1em]
PSD cSVD &
  $\fV_{2k} = [\fE(:,1:k) \quad \TJtn \fE(:,1:k)]$ &
  $\fE = \begin{bmatrix}\fPhi \\ \fPsi \end{bmatrix}, \fPhi +\imagUnit\fPsi = \prcCSVD\lb \Cs \rb$ &
  \cmark &
  \cmark\\
&& $\Cs = [\fp_1 + \imagUnit \fq_1, \dots, \fp_{\ns} + \imagUnit \fq_{\ns}]$\\[1em]
PSD greedy &
  $\fV_{2k} = [\fE(:,1:k) \quad \TJtn \fE(:,1:k)]$ &
  $\fE$ from greedy algorithm &
  \cmark &
  \cmark\\[1em]
PSD SVD-like &
  $\fV_{2k} = [\fs_{i_1}, \dots, \fs_{i_k}, \fs_{n+i_1}, \dots, \fs_{n+i_k}]$ &
  $\fS = [\fs_1, \dots, \fs_{2n}]$ from \cref{eq:SVDlikeDecomp}, &
  \xmark &
  \cmark\\
&& $\idxSetPSD = \{i_1, \dots, i_k\}$ from \cref{eq:IdxPSDSVDlike}\\
\hline\noalign{\smallskip}
\end{tabular}
}
\caption{Basis generation methods used in the numerical experiments in summary, where we use the MATLAB\textsuperscript{\textregistered} notation to denote the selection of the first $k$ columns of a matrix e.g.\ in $\fU(:,1:k)$.}
\label{tab:basis_gen}
\end{table}

All presented experiments are generalization experiments, i.e.\ we choose $9$ different training parameter vectors $\fmu \in \paramSet$ on a regular grid to generate the snapshots and evaluate the reduced models for $16$ random parameter vectors that are distinct from the $9$ training parameter vectors. Thus, the number of snapshots is $\ns = 9 \cdot 151 = 1359$. The size $2k$ of the ROB $\fV$ is varied in steps of $20$ with $2k \in \lcb 20, 40, \dots, 280, 300 \rcb$.

The software used for the numerical experiments is RBmatlab\footnote{https://www.morepas.org/software/rbmatlab/} which is an open-source library based on the proprietary software package MATLAB\textsuperscript{\textregistered} and contains several reduced simulation approaches. An add-on to RBmatlab is provided\footnote{https://doi.org/10.5281/zenodo.2578078} which includes all the additional code to reproduce the results of the present paper. The versions used in the present paper are RBmatlab 1.16.09 and MATLAB\textsuperscript{\textregistered} 2017a.

\subsection{Autonomous beam model}
In the first model, we load the beam on the free end (far right) with a constant force which induces an oscillation.  Due to the constant force, the discretized system can be formulated as an autonomous Hamiltonian system. Thus, the Hamiltonian is constant and its preservation in the reduced models can be analysed. All other reduction results are very similar to the non-autonomous case and thus, are exclusively presented for the non-autonomous case in the following \cref{sec:NonAutoBeam}.

\subsubsection{Preservation over time of the modified Hamiltonian in the reduced model}

In the following, we investigate the preservation of the Hamiltonian of our reduced models. With respect to \cref{rem:ModHam}, we mean the preservation over time of the modified Hamiltonian. Since the Hamiltonian is quadratic in our example and the implicit midpoint is a symplectic Runge-Kutta integrator, the modified Hamiltonian equals the original which is why we speak of ``the Hamiltonian'' in the following.

We present in \cref{fig:NumRes:BeamAuto:Ex2:Hamiltonian} the count of the total $240$ simulations which show a preservation (over time) of the reduced Hamiltonian in the reduced model. The solution $\fxr$ of a reduced simulation preserves the reduced Hamiltonian over time if $\txtfrac{\lb \Hamr(\fxr(t_i), \fmu) - \Hamr(\fxr(t_0), \fmu) \rb}{\HamRel(\fmu)} < 10^{-10}$ for all discrete times $t_i \in [\tInit, \tEnd]$, $1\leq i\leq \nt$ where $\HamRel(\fmu) > 0$ is a parameter-dependent normalization factor. The heat map shows that no simulation in the non-symplectic case preserves the Hamiltonian whereas the symplectic methods all preserve the Hamiltonian which is what was expected from theory.

In \cref{fig:NumRes:BeamAuto:Ex2:HamExample}, we exemplify the non-constant evolution of the reduced Hamiltonian for three non-symplectic bases generated by POD separate states with different basis sizes and one selected test parameter $(\lambda, \mu) \in \paramSet$. It shows that in all three cases, the Hamiltonian starts to grow exponentially.

\begin{figure}[h]
\begin{minipage}{.47\textwidth}
\centering
\includegraphics[scale=1.0]{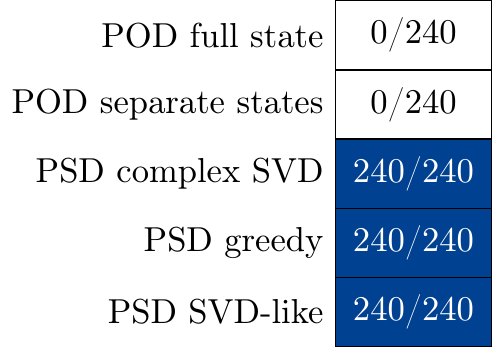}
\caption{Heat map which shows the preservation of the reduced Hamiltonian in the reduced model in $x$ of $y$ cases ($x/y$).}
\label{fig:NumRes:BeamAuto:Ex2:Hamiltonian}
\end{minipage}\hfill
\begin{minipage}{.47\textwidth}
\centering
\includegraphics[scale=1.0]{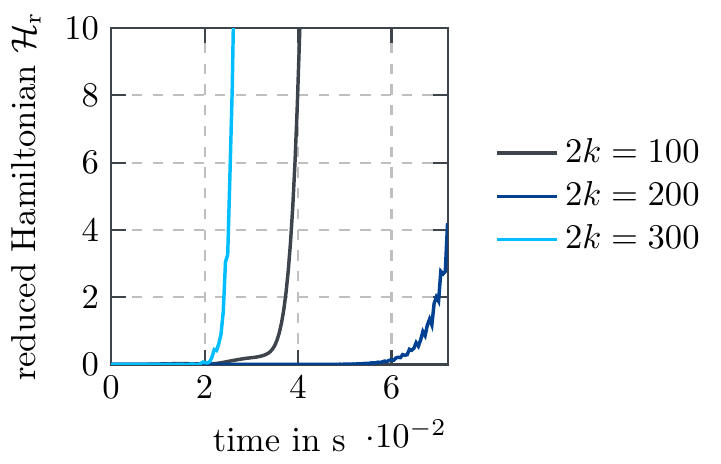}
\caption{Evolution of the reduced Hamiltonian for POD separate states for a selected parameter $(\lambda, \mu) \in \paramSet$.}
\label{fig:NumRes:BeamAuto:Ex2:HamExample}
\end{minipage}
\end{figure}

\subsection{Non-autonomous beam model}
\label{sec:NonAutoBeam}
The second model is similar to the first one. The only difference is that the force on the free (right) end of the beam is loaded with a time-varying force. The force is chosen to act in phase with the beam. The time dependence of the force necessarily requires a non-autonomous formulation which requires in the framework of the Hamiltonian formulation a time-dependent Hamiltonian function which we introduced in \cref{subsec:NonAutoHam}.

We use the model to investigate the quality of the reduction for the considered MOR techniques. To this end, we investigate the projection error, i.e.\ the error on the training data, the orthogonality and symplecticity of the ROB and the error in the reduced model for the test parameters.

\subsubsection{Projection error of the snapshots and singular values}
\label{subsubsec:NumExp:projErr}
The projection error is the error on the training data collected in the snapshot matrix $\Xs$, i.e.\
\begin{align*}
	&\errProj(2k) = \Fnorm{(\I{2n} - \fV \rT\fW) \Xs}^2,&
	\begin{split}
    \text{POD}: \rT\fW =&\; \rT\fV,\\
    \text{PSD}: \rT\fW =&\; \si\fV (= \rT\fV \text {for orthosymplectic ROBs, \cref{lem:SymplOrthonROB}}).
  \end{split}
\end{align*}
It is a measure for the approximation qualities of the ROB based on the training data. \cref{fig:NumRes:BeamNonAuto:Ex3:BasisProps} (left) shows this quantity for the considered MOR techniques and different ROB sizes $2k$. All basis generation techniques show an exponential decay. As expected from theory, POD full state minimizes the projection error for the orthonormal basis generation techniques (see \cref{tab:basis_gen}). PSD SVD-like decomposition shows a lower projection error than the other PSD methods for $2k \geq 80$ and yields a similar projection error for $k \leq 60$. Concluding this experiment, one might expect the full-state POD to yield decent results or even the best results. The following experiments prove this expectation to be wrong.

The decay of (a) the classical singular values $\sigma_i$, (b) the symplectic singular values $\sigmaS_i$ (see \cref{rem:SingVal}) and (c) the weighted symplectic singular values $\wsigmaS_i$ (see \cref{eq:FrobSymplSingVal}) sorted by the magnitude of the symplectic singular values  is displayed in \cref{fig:NumRes:BeamNonAuto:Ex3:BasisProps} (right). All show an exponential decrease. The weighting introduced in \cref{eq:FrobSymplSingVal} for $\wsigmaS_i$ does not influence the exponential decay rate of $\sigmaS_i$. The decrease in the classical singular values is directly linked to the exponential decrease of the projection error of POD full state due to properties of the Frobenius norm (see \cite{LuminyBook2017}). A similar result was deduced in the scope of the present paper for PSD SVD-like decomposition and the PSD functional (see \cref{prop:PSDDecay}).

\begin{figure}[h]
\centering
\includegraphics[scale=1.0]{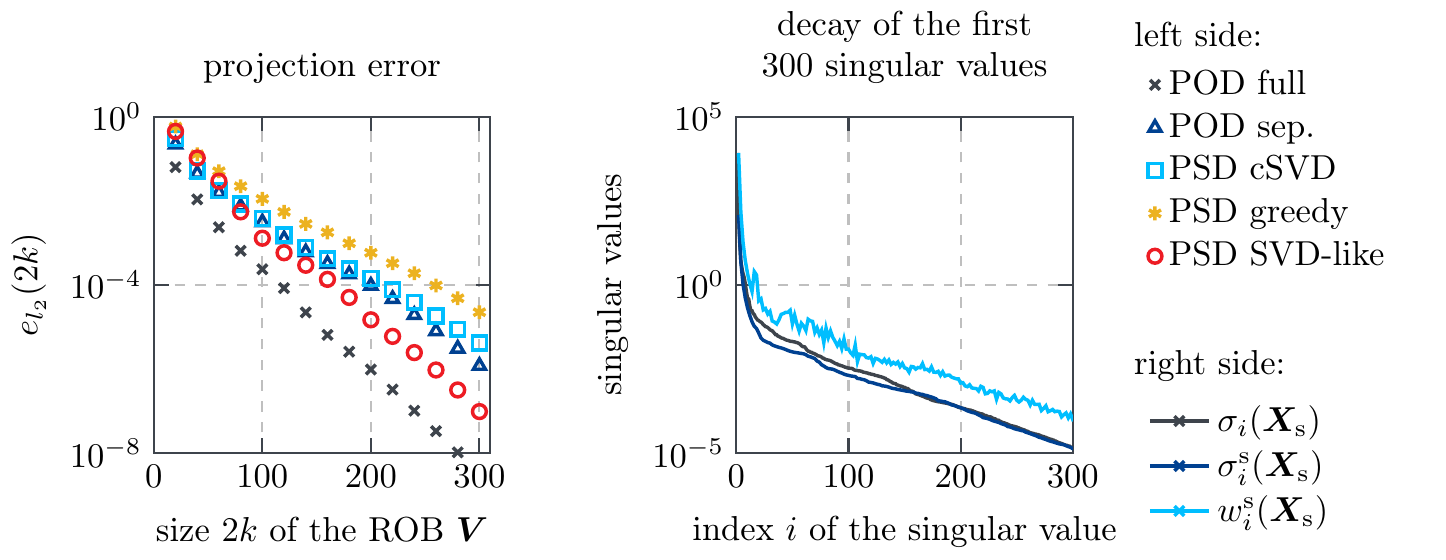}
\caption{Projection error (left) and decay of the singular values from \cref{rem:SingVal} and \cref{eq:FrobSymplSingVal} (right).}
\label{fig:NumRes:BeamNonAuto:Ex3:BasisProps}
\end{figure}

\subsubsection{Orthonormality and symplecticity of the bases}
To verify the orthonormality and the symplecticity numerically, we consider the two functions
\begin{align} \label{eq:ortho}
  &\orthoMeasure(2k) = \Fnorm{\rT\fV \fV - \I{2k}},&
  &\symMeasure(2k) = \Fnorm{\rT\Jtk \rT\fV \Jtn \fV - \I{2k}}&
\end{align}
which are zero / numerically zero if and only if the basis is orthonormal or symplectic, respectively. In \cref{fig:NumRes:BeamAuto:Ex1:Ortho}, we show both values for the considered basis generation techniques and RB sizes.

The orthonormality of the bases is in accordance with the theory. All procedures compute symplectic bases except for PSD SVD like-decomposition. PSD greedy shows minor loss in the orthonormality which is a known issue for the $\Jtn$-orthogonalization method used (modified symplectic Gram-Schmidt procedure with re-orthogonalization \cite{AlAidarous2011}). But no major impact on the reduction results could be attributed to this deficiency in the scope of this paper.

Also the symplecticity (or $\Jtn$-orthogonality) of the bases behaves as expected. All PSD methods generate symplectic bases whereas the POD methods do not. A minor loss of symplecticity is recorded for PSD SVD-like decomposition which is objected to the computational method that is used to compute an SVD-like decomposition. Further research on algorithms for the computation of an SVD-like decomposition should improve this result. Nevertheless, no major impact on the reduction results could be attributed to this deficiency in the scope of this paper.

\begin{figure}[h]
\centering
\includegraphics[scale=1.0]{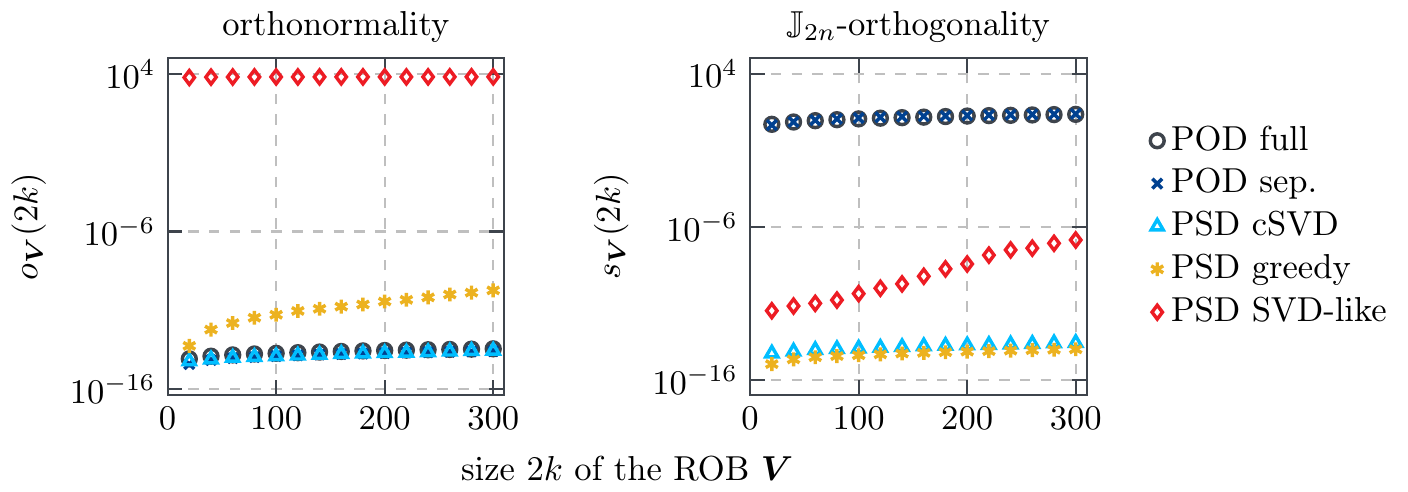}
\caption{The orthonormality (left) and the $\Jtn$-orthogonality (right) from \cref{eq:ortho}.}
\label{fig:NumRes:BeamAuto:Ex1:Ortho}
\end{figure}

\subsubsection{Relative error in the reduced model}
We investigate the error introduced by MOR in the reduced model. The error is measured in the relative $\infty$-norm $\infnorm{\bullet}$ in time and space
\begin{align}\label{eq:ErrMOR}
  \meanErr(2k, \fmu) :=
\frac{\displaystyle\max_{i\in\lcb 1,\dots,\nt \rcb}  \infnorm{\fx(t_i, \fmu) - \fV \fxr(t_i, \fmu)}}
{\displaystyle\max_{i\in\lcb 1,\dots,\nt \rcb} \infnorm{\fx(t_i, \fmu)}},
\end{align}
where $2k$ indicates the size of the ROB $\fV \in \R^{2n \times 2k}$ and $\fmu \in \paramSet$ is one of the test parameters. To display the results for all $16$ test parameters at once, we use box plots in \cref{fig:NumRes:BeamNonAuto:Ex4:RelErr}. The box represents the $25\%$-quartile, the median and the $75\%$-quartile. The whiskers indicate the range of data points which lay within $1.5$ times the interquartile range (IQR). The crosses show outliers. For the sake of a better overview, we truncated relative errors above $10^0 = 100\%$.

The experiments show that the non-symplectic MOR techniques show a strongly non-monotonic behaviour for an increasing basis size. For many of the basis sizes, there exists a parameter which shows crude approximation results which lay above $100\%$ relative error. The POD full state is unable to produce results with a relative error below $2\%$.

On the other hand, the symplectic MOR techniques show an exponentially decreasing relative error. Furthermore, the IQRs are much lower than for the non-symplectic methods. We stress that the logarithmic scale of the $y$ axis distorts the comparison of the IQRs -- but only in favour of the non-symplectic methods. The low IQRs for the symplectic methods show that the symplectic MOR techniques derive a reliable reduced model that yields good results for any of the $16$ randomly chosen test parameters. Furthermore, none of the systems shows an error above $0.19\%$ -- for PSD SVD-like decomposition this bound is $0.018\%$, i.e.\ one magnitude lower.

In the set of the considered symplectic, orthogonal MOR techniques, PSD greedy shows the best result for most of the considered ROB sizes. This superior behaviour of PSD greedy in comparison to PSD complex SVD is unexpected since PSD greedy showed inferior results for the projection error in \cref{subsubsec:NumExp:projErr}. This was also observed in \cite{Maboudi2017}.

Within the set of investigated symplectic MOR techniques, PSD SVD-like decomposition shows the best results followed by PSD greedy and PSD complex SVD. While the two orthonormal procedures show comparable results, PSD SVD like-decomposition shows an improvement in the relative error. Comparing the best result of either PSD greedy or PSD complex SVD with the worst result of PSD SVD-like decomposition considering the $16$ different test parameters for a fixed basis size -- which is pretty much in favour of the orthonormal basis generation techniques --, the improvement of PSD SVD-like decomposition ranges from factor $3.3$ to $11.3$ with a mean of $6.7$.

\begin{figure}[h]
\centering
\includegraphics[scale=1.0]{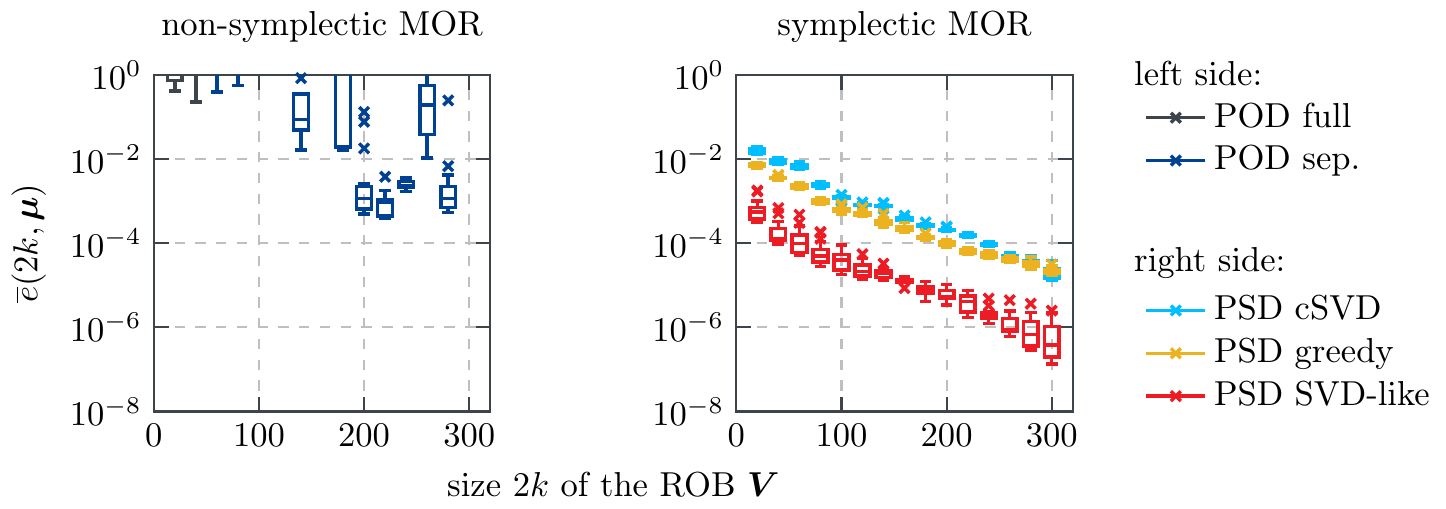}
\caption{Relative error in the reduced model.}
\label{fig:NumRes:BeamNonAuto:Ex4:RelErr}
\end{figure}

%
\section{Summary and conclusions}
\label{sec:Conclusion}

We gave an overview of autonomous and non-autonomous Hamiltonian systems and the structure-preserving model order reduction (MOR) techniques for these kinds of systems \cite{Peng2016,Maboudi2017,Maboudi2018}. Furthermore, we classified the techniques in orthonormal and non-orthonormal procedures based on the capability to compute a symplectic, (non-)orthonormal reduced order basis (ROB). To this end, we introduced a characterization of rectangular, symplectic matrices with orthonormal columns. Based thereon, an alternative formulation of the PSD Complex SVD \cite{Peng2016} was derived which we used to prove the optimality with respect to the PSD functional in the set of orthonormal, symplectic ROBs. As a new method, we presented a symplectic, non-orthonormal basis generation procedure that is based on an SVD-like decomposition \cite{Xu2003}. First theoretical results show that the quality of approximation can be linked to a quantity we referred to as weighted symplectic singular values.

The numerical examples show advantages for the considered linear elasticity model for symplectic MOR if a symplectic integrator is used. We were able to reduce the error introduced by the reduction with the newly introduced non-orthonormal method.

We conclude that non-orthonormal methods are able to derive bases with a lower error for both, the training and the test data. Yet, it is still unclear if the newly introduced method computes the global optimum of the PSD functional. Further work should investigate if a global optimum of the PSD functional can be computed with an SVD-like decomposition.\\

%
\vspace{6pt} 

\funding{This research was partly funded by the German Research Foundation (DFG) grant number HA5821/5-1 and within the GRK 2198/1.}

\acknowledgments{We thank the German Research Foundation (DFG) for funding this work. The authors thank Dominik Wittwar for inspiring discussions.}

\bibliographystyle{plain}
\bibliography{symplectic,port_Hamiltonian,mor,specific_bib}



\end{document}